%% file: siam-reconstruction.tex
\pdfoutput=1
\documentclass[final]{siamart0516}

\input{rec_packages}
\input{rec_commands}

\input{rec_title}

\ifpdf
\hypersetup{
  pdftitle={\TheTitle},
  pdfauthor={\TheAuthors}
}
\fi

\begin{document}

\maketitle

\begin{abstract}
Classical inf-sup stable mixed finite elements for the incompressible (Navier--)Stokes equations are not pressure-robust, i.e.,
their velocity errors depend on the continuous pressure. However, a modification only in the right hand side of a Stokes discretization is able to reestablish pressure-robustness,
as shown recently for several inf-sup stable Stokes elements with discontinuous discrete pressures.
In this contribution, this idea is extended to low and high order Taylor--Hood and mini elements, which have continuous discrete pressures.
For the modification of the right hand side a velocity reconstruction operator is constructed that maps discretely divergence-free
test functions to exactly divergence-free ones. The reconstruction is based on local $H(\mathrm{div})$-conforming flux equilibration on vertex patches,
and fulfills certain orthogonality properties to provide consistency and optimal a-priori error estimates.
Numerical examples for the incompressible
Stokes and Navier--Stokes equations confirm that the new pressure-robust Taylor--Hood and mini elements
converge with optimal order and outperform significantly
the classical versions of those elements when the continuous pressure is comparably large.
\end{abstract}

\begin{keywords}
  incompressible Navier--Stokes equations, mixed finite elements, pressure robustness, exact divergence-free velocity reconstruction,
  flux equilibration
\end{keywords}

\begin{AMS}
  65N12, 65N30, 76D07, 76D05, 76M10
\end{AMS}

\section{Introduction and notation}
\input{rec_intro}

\section{Continuous and discrete Stokes problems and the velocity reconstruction operator}
\input{rec_method}

\section{Error estimation for the pressure-robust Stokes discretization} \label{sec::errorest}
\input{rec_errorest}

\section{Construction and analysis of the reconstruction operator} \label{sec::ana}
\input{rec_analysis}

\section{Numerical examples}
\input{rec_numex}

\section{Appendix}
\input{rec_appendix}

\bibliographystyle{siamplain}
\bibliography{rec_references}
\end{document}

%% file: rec_packages.tex
\usepackage{lipsum}
\usepackage{amsfonts}
\usepackage{graphicx}
\usepackage{epstopdf}
\usepackage{amsmath,amssymb}
\usepackage{algorithmic}
\ifpdf
  \DeclareGraphicsExtensions{.eps,.pdf,.png,.jpg}
\else
  \DeclareGraphicsExtensions{.eps}
\fi

\usepackage{amsopn}

\usepackage{todonotes}

\usepackage{booktabs,colortbl}
\usepackage{pgfplots, pgfplotstable}
\usepackage{tikz}
\usepackage{graphicx}
\usepackage{float}
\usepackage{subfig}
\usepackage{tikz}
\usetikzlibrary{patterns,calc,shapes,arrows,plotmarks,spy}

%% file: rec_commands.tex
\newcommand{\R}{\mathbb{R}}

\newcommand{\Tc}{\mathcal{T}}

\newcommand{\Pc}{\mathcal{P}}

\newcommand{\cle}{ \preccurlyeq }
\newcommand{\cge}{ \succcurlyeq }
\newcommand{\Tref}{\hat{T}}

\newcommand{\norm}[2]{\left\|#1\right\|_{#2}}

\newcommand{\laplace}[1]{\Delta #1}

\newcommand{\andtext}{\textrm{and}}
\newcommand{\vg}{\textbf{g}}
\newcommand{\vq}{\textbf{q}}
\newcommand{\vu}{\textbf{u}}
\newcommand{\vv}{\textbf{v}}
\newcommand{\vw}{\textbf{w}}
\newcommand{\vzero}{\textbf{0}}
\newcommand{\bfe}{\textbf{f}}

\newcommand{\vV}{\mathrm{{\bf V}}}

\newcommand{\TcV}{\Tc_{\omega_V}}

\newcommand{\intO}{\int_\Omega}

\newcommand{\Hdiv}[1]{H(\textrm{\textup{div}},{#1})}

\newcommand{\Hdivoz}[1]{H_0(\textrm{\textup{div}},{#1})}

\newcommand{\trace}{\textrm{\textup{tr}}}
\newcommand{\tracen}{\textrm{\textup{tr}}_{\boldsymbol{n}}}

\newcommand{\RT}[1]{\mathcal{RT}^{#1}}
\newcommand{\IRT}{I_\mathcal{RT}}

\newcommand{\pol}{\Pi}

\newcommand{\ondom}[1]{\textrm{~on~} #1}

\renewcommand{\theequation}{\thesection.\arabic{equation}}
\numberwithin{equation}{section}

\renewcommand{\epsilon}{\varepsilon}

\let\div\undefined\DeclareMathOperator{\div}{div}

\newcommand{\curl}[1]{\textrm{\textup{curl}~} #1}
\newcommand{\Curl}[1]{\textrm{\textup{Curl}~} #1}

\newcommand{\intd}{~\mathrm{d}}

\newcommand{\Rec}{\mathcal{R}_h}

\newcommand{\Vh}{\vV_h}
\newcommand{\Wh}{W_h}
\newcommand{\w}{\boldsymbol{w}}

\newcommand{\blfB}{\mathcal{B}}
\newcommand{\Sigh}{\boldsymbol{\Sigma}_h}
\newcommand{\Sigho}{\Sigma_{h,0}}

\newcommand{\sig}{\boldsymbol{\sigma}_h}
\newcommand{\sigb}{\boldsymbol{\sigma}}
\newcommand{\sigV}{\boldsymbol{\sigma}^V_h}

\newcommand{\Qh}{Q_h}
\newcommand{\tQh}{\widetilde{Q}_h}
\newcommand{\tqh}{\tilde{q}_h}

\renewcommand{\a}{\boldsymbol{\xi}}
\renewcommand{\b}{b}

\renewcommand{\Vert}{\mathcal{V}}

\newcommand{\ov}{\omega_V}
\newcommand{\vertiii}[1]{{\left\vert\kern-0.25ex\left\vert\kern-0.25ex\left\vert #1 
    \right\vert\kern-0.25ex\right\vert\kern-0.25ex\right\vert}}
\newcommand{\normsmooth}[1]{\vertiii{#1}}
\newcommand{\normlo}[1]{\norm{#1}{L^2(\ov)}}

\newcommand{\Projlktwo}{\mathcal{P}^{k-2}_{\ov}}

\newcommand{\skp}[1]{\left( #1 \right)_{L^2(\Omega)}}
\newcommand{\skpV}[1]{\left( #1 \right)_{L^2(\ov)}}

\newcommand{\into}{\int_{\ov}}
\newcommand{\newvec}[2]{\begin{pmatrix}#1\\#2\end{pmatrix}}

\newcommand{\oswald}{\mathcal{S}}
\newcommand{\oswaldt}{\tilde{\mathcal{S}}}
\newcommand{\restr}[2]{\ensuremath{\left.#1\right|_{#2}}}
\newcommand{\bubb}{\mathcal{P}^\mathcal{B}}

\newcommand{\normsigo}[1]{\norm{#1}{\Sigho}}
\newcommand{\normtqo}[1]{\norm{#1}{\tQh}}
\newcommand{\normwo}[1]{\norm{#1}{\Wh}}
\newsiamremark{remark}{Remark}

\newcommand{\kos}{\kappa}
\newcommand{\fortin}{I_{F}}

%% file: rec_title.tex
\newcommand{\TheTitle}{Divergence-free Reconstruction Operators for Pressure-Robust Stokes Discretizations With Continuous Pressure Finite Elements}
\newcommand{\TheAuthors}{P. L. Lederer, A. Linke, C. Merdon, J. Sch\"oberl}

\title{{\TheTitle}}

\author{
  Philip L. Lederer\thanks{Institute for Analysis and Scientific Computing, TU Wien, Austria
    (\email{philip.lederer@tuwien.ac.at},\email{joachim.schoeberl@tuwien.ac.at} ).}
  \and
  Alexander Linke\thanks{Weierstrass Institute for Applied Analysis and Stochastics, Germany
    (\email{alexander.linke@wias-berlin.de}, \email{christian.merdon@wias-berlin.de}).}
  \and
  Christian Merdon\footnotemark[3]
  \and
  Joachim Sch\"oberl\footnotemark[2]
}

%% file: rec_intro.tex
% SIAM Shared Information Template
% This is information that is shared between the main document and any
% supplement. If no supplement is required, then this information can
% be included directly in the main document.

\subsection{Introduction}
The classical Taylor--Hood element \cite{verf:1984, girault:raviart, msw:2013}, its higher order extensions \cite{brezzi:falk} and the classical mini element
\cite{abf:1984, girault:raviart} are among the most popular
discretizations for the incompressible Navier--Stokes equations, since they are easy to implement, fulfill a discrete
LBB condition and converge with optimal order.
Nevertheless they suffer from a common lack of robustness: since they use continuous discrete pressures, they relax the divergence constraint
and are thus not pressure-robust \cite{JLMNR:sirev}, i.e., their velocity error is pressure-dependent, as one can see for an incompressible Stokes model problem
$-\nu \Delta \vu + \nabla p = \bfe, \div \vu = 0$
with homogeneous Dirichlet velocity boundary conditions (with $\nu > 0$).
Here, the velocity errors for the Taylor--Hood and mini elements read as
$$
  \norm{\nabla (\vu - \vu_h)}{L^2(\Omega)} \leq C \inf_{\vw_h \in \Vh} \norm{\nabla (\vu - \vw_h)}{L^2(\Omega)} + \frac{1}{\nu} \inf_{q_h \in Q_h}\norm{p - q_h}{L^2(\Omega)},
$$
where $\Vh$ and $Q_h$ denote the discrete trial/test spaces for the velocities and the pressures, and $C$ is a $\mathcal{O}(1)$ constant.
This velocity error estimate is sharp and shows some kind of locking phenomenon \cite{JLMNR:sirev, linke:cmame:2014, reusken, FH88, olhl2009}:
for small parameters $\nu \ll 1$ the velocity error can become really large.
The issue is well-known in the literature, it shows up in real-world situations \cite{DGT94, gerbeau:1997, LM2015, JLMNR:sirev} and it is sometimes called
{\em poor mass conservation} \cite{GLRW:2012},
since for $H^1$-conforming mixed methods such large velocity errors are accompanied by large divergence errors.

Recently, it was shown for several mixed finite element methods
like the nonconforming Crouzeix--Raviart element \cite{linke:cmame:2014, blms:2015}
and the conforming $P_2^+$-$P_1^{\mathrm{disc}}$ element \cite{lmt:2016}
(and also for a finite volume \cite{Linke:2012} and for some Hybrid Discontinuous Galerkin methods \cite{dPEL:2016}), which
all use {\em discontinuous pressures},  that a modification only in the right hand side of the Stokes discretization
is able to reestablish pressure-robustness. This approach leads to a velocity error estimate \cite{linke:cmame:2014, lmt:2016}
$$
  \norm{\nabla (\vu - \vu_h)}{L^2(\Omega)} \leq C \inf_{\vw_h \in \Vh} \norm{\nabla (\vu - \vw_h)}{L^2(\Omega)} + C_{\mathrm{cons}} h^{l+1} \lvert \vu \rvert_{H^{l+1}(\Omega)},
$$
where $l$ denotes the approximation order of the discrete pressure space and $C_{\mathrm{cons}}$ denotes an $\mathcal{O}(1)$ constant, arising due to a consistency error in the
discrete right hand side.
Note that similar pressure-robust velocity error estimates can be achieved also with {\em divergence-free} mixed methods like
\cite{qin:phd, zhang:scott:vogelius:3D, zhang:2009, guzman:neilan:3d, guzman:neilan:2014, ls:2016, guosheng}.
The key idea for the modification of the Stokes right hand side in \cite{linke:cmame:2014} is that discrete divergence-free velocity test functions
are mapped to exact divergence-free ones by some velocity reconstruction operator. Then, irrotational parts (in the sense of the
continuous Helmholtz decomposition) in the exterior force $\bfe$ of the
above Stokes model problem are orthogonal in the $L^2$ vector product
to (mapped) discrete-divergence velocity test functions and do not spoil the discrete velocity solution $\vu_h$ \cite{linke:cmame:2014}.
Indeed, the so-called {\em poor mass conservation} arises just due to a lack of $L^2$ orthogonality between discrete-divergence-free velocity
test functions and {\em arbitrary} gradient fields $\nabla \psi$ \cite{Linke:2012, linke:cmame:2014, JLMNR:sirev}.
For LBB-stable mixed finite element methods with discontinuous pressures the corresponding velocity reconstruction operators
employ $H(\mathrm{div})$-conforming finite element spaces. The velocity reconstruction operator is defined elementwise, and fulfills
several consistency properties \cite{lmt:2016}.

At the heart of the present contribution lies the construction of novel velocity reconstruction operators for the Taylor--Hood element family and the mini element, which have
continuous discrete pressures, such that a modification of the Stokes right hand side yields a pressure-robust mixed method.
A first version of such velocity reconstruction operators has been presented in \cite{lederer:2016}. Similarly, velocity reconstructions in the spirit of \cite{gww:2014}
could be probably adapted also.
Since the new corresponding mixed
methods have the same stiffness matrix like their classical counterparts, the discrete LBB condition is inherited from the original method.
%Numerical benchmarks show that these pressure-robust modifications of the Taylor--Hood and the mini elements outperform the original methods clearly,
%when $\nu$ is small and when the continuous pressure is non-trivial --- in the sense that the pressure \(p \notin Q_h\) and the
%velocity $\vu$ are of comparable magnitude.
Optimal convergence of the new pressure-robust mixed methods is shown.
The novel velocity reconstructions require the solution of local discrete problems, which are defined on vertex patches.
The reconstructions map $H^1$-conforming velocity test functions to $H(\mathrm{div})$-conforming ones, which preserve the discrete divergence.
Especially, discrete divergence-free velocities are mapped to exact divergence-free ones. The construction uses ideas from flux equilibration for
a-posteriori estimates \cite{dm:1999, bs:2008}. In order to achieve optimal convergence order for the novel mixed methods, the velocity reconstructions have to fulfill
some consistency properties, which are incorporated in the local problems to be solved. For this, bubble projectors \cite{bubbletransform}, averaging operators \cite{oswald}
and properties of the Koszul complex \cite{koszul} have to be exploited.

\subsection{Structure of this paper}
After defining some notation in the next subsection, in Section 2 the continuous Stokes problem
is introduced and the new pressure-robust mixed finite element methods for its discretizations are presented in a quite abstract manner.
The main Theorem \ref{theoremone} summarizes the most important properties
of the velocity reconstruction operator $\Rec$, while the proofs of these properties
are postponed to Section 4 in case of the Taylor--Hood element family and to Section 5 in case of the mini element.
Section \ref{sec::errorest} presents a common finite element error analysis for the proposed Taylor--Hood and mini element variants. It is shown that
their velocity errors are indeed pressure-robust, and that --- quite surprisingly ---
even pressure-robustness results hold for their pressure errors, when measured in some {\em discrete} pressure norms.
In Section 4, different finite element spaces and finite element tools like bubble projectors \cite{bubbletransform}
and Oswald interpolators are introduced, and local (saddle-point) problems on vertex patches are defined
that are fundamental for the definition of the novel velocity reconstruction operators for the Taylor--Hood finite element family.
Besides proving the unique solvability of these local problems, the properties of the corresponding reconstruction operators stated in
Theorem \ref{theoremone} are proved.
Similar to Section 4, in Section 5 velocity reconstruction operators for lowest and higher order mini elements are defined solving local problems on vertex patches,
and the properties of Theorem \ref{theoremone} are proved also in these cases.
Section 6 presents several numerical examples for the incompressible Stokes equations in 2D and 3D that show that the pressure-robust Taylor--Hood and mini element
variants can outperform clearly their classical counterparts in the best case, and are only slightly worse than the classical discretizations in the worst case.
Section 7 serves as an Appendix where some properties of the Koszul complex in 3D are demonstrated.

\subsection{Preliminaries}
We introduce some basic notation and assumptions. In this work we assume an open bounded domain $\Omega \subset \R^d$ with $d = 2,3$ and a Lipschitz boundary $\Gamma$. On $\Omega$ we define a partition $\Omega = \bigcup_{i=1}^{N_T} T_i$ into sub-domains called elements $T_i$ which will be triangles and tetrahedrons in two and three dimensions respectively. We shall denote $\Tc$ as such a partition which fulfills a shape regular assumption, so all elements fulfill $|T| \cge \textrm{diam}(T)^d$. Furthermore we call $\Tc$ quasi--uniform when all elements are essentially of the same size, i.e., there exists one global $h$ such that $h \approx \textrm{diam}(T),  \forall T \in \Tc$, see for example \cite{brezzifortin}. The set of vertices is defined as $\Vert$ and for each vertex $V \in \Vert$ we define the vertex patch $\ov$ and the corresponding triangulation $\TcV$ as 
\begin{align*}
  \ov:= \bigcup\limits_{T: V \in T} T \subset \Omega \quad \andtext \quad \TcV := \{T: V \in T \} \subset \Tc,
\end{align*}
and define the local mesh size $h_V := \max \{\textrm{diam}(T): T \in \TcV \}$. We define the polynomial spaces of order $m$ on $\Omega$ as $\pol^m(\Omega)$ and on the triangulation as
\begin{align}
    \pol^m(\Tc) := \{ q_h : \restr{q_h}{T} \in \pol^m(T)~ \forall T \in \Tc \} = \prod_{ T \in \Tc} \pol^m(T),
\end{align}
and similar for $\ov$ and $\TcV$. Furthermore we define the spaces
\begin{align*}
  L^2_0(\Omega) &:= \{ q \in L^2(\Omega): \int_\Omega q \intd x = 0\} =: Q, \\
  H^1_0(\Omega) &:= \{ u \in H^1(\Omega): \trace ~ u = 0 \ondom \partial \Omega\}, \\
  \Hdivoz{\Omega}  &:= \{ \sigb \in \Hdiv{\Omega}: \tracen \sigma = 0 \ondom \partial \Omega \}, \\
  \vV & := [H^1_0(\Omega)]^d, \\
  \vV^0 & := \{ \vv \in \vV : \div \vv = 0 \}.
\end{align*}
where $\trace$ and $\tracen$ denote the trace operators for $H^1(\Omega)$ and $\Hdiv{\Omega}$. We also define the $L^2$ projector on polynomials of order $m$ as $\mathcal{P}^m_{\Omega}$, and the Oswald interpolator $\oswald: \pol^m(\Tc) \rightarrow \pol^m(\Tc) \cap C^0(\Omega)$ (see \cite{oswald} or the averaging operator in \cite{ernguermond15}) that maps discontinuous polynomials to continuous ones.
Depending on the dimension we define the Koszul operator  (see \cite{koszul}) for $d=2$ with $\vec{x}=(x,y)$ and for $d = 3$ with $\vec{x}=(x,y,z)$ as
\begin{align*}
\begin{array}{lcccl}
\kos_{\vec{x}} : L^2(\Omega) \rightarrow  [L^2(\Omega)]^2 &\quad &\quad&\quad & \kos_{\vec{x}} : [L^2(\Omega)]^3 \rightarrow  [L^2(\Omega)]^3 \\
  \kos_{\vec{x}}(a) :=   \newvec{-y}{x} a &\quad &\quad&\quad &\kos_{\vec{x}}(a) :=  \vec{x} \times a.
\end{array}
\end{align*}
Furthermore we define the {\it Curl} operator for $d=2$ 
\begin{align*}
  &\Curl: \pol^m(\Omega) \rightarrow [\pol^m(\Omega)]^2 \\
  &\Curl{(u)} := (-\partial_y u, \partial_x u)^t.
\end{align*}
In a similar way all the above introduced spaces and operators can be defines on $\ov$. In this work we use $a \cle b$ when there exists a constant $c$ independent of $a, b, m, h$  such that $a \le c b$

%%% Local Variables: 
%%% mode:latex
%%% TeX-master: "siam-reconstruction"
%%% End: 

%% file: rec_method.tex
% SIAM Shared Information Template
% This is information that is shared between the main document and any
% supplement. If no supplement is required, then this information can
% be included directly in the main document.
The incompressible Stokes problem for a right hand side forcing $\bfe \in [L^2(\Omega)]^d$ is given
in weak formulation by \cite{girault:raviart}: search for $(\vu, p) \in \vV \times Q$
such that for all $(\vv, q) \in \vV \times Q$ holds
\begin{equation}
\begin{split}
   a(\vu, \vv) + b(\vv, p) & = l(\vv), \label{cont:weak:stokes:problem} \\
     b(\vu, q) & = 0,
\end{split}
\end{equation}
where the bilinear forms $a: \vV \times \vV \to \R$
and $b: \vV \times Q \to \R$ and the linear form
$l: [L^2(\Omega)]^d \to \R$ are defined by
\begin{equation}
\begin{split}
  a(\vu, \vv) & =  \intO \nu \nabla \vu : \nabla \vv \intd x, \\
  b(\vv, q) & =  \intO q \, \div \vv \intd x, \\
  l(\vv) & = \intO \bfe \cdot \vv \intd x.
\end{split}
\end{equation}
Note that for the continuous Stokes problem holds the LBB condition
\begin{equation} \label{contstokesLBB}
 \inf_{q \in Q} \sup_{\vv \in \vV} \frac{b(\vv, q)}{\norm{q}{L^2(\Omega)} \norm{\nabla \vv}{L^2(\Omega)}} \geq \beta > 0,
\end{equation}
where $\beta$ denotes the LBB constant.

For the discretization of the continuous Stokes problem \eqref{cont:weak:stokes:problem} by inf-sup stable mixed finite element methods \cite{girault:raviart, brezzifortin}
we introduce conforming finite element spaces for the velocity $\vV_h \subset \vV$ and the pressure $Q_h \subset Q$. We assume that for the pair $\vV_h \times Q_h$
of discrete spaces holds a discrete LBB condition
\begin{equation} \label{discrete:LBB:cond}
 \inf_{q_h \in Q_h} \sup_{\vv_h \in \vV_h} \frac{b(\vv_h, q_h)}{\norm{q_h}{L^2(\Omega)} \norm{\nabla \vv_h}{L^2(\Omega)}} \geq \beta_h > 0.
\end{equation}
We remind the reader that the discrete LBB condition implies the existence of a Fortin interpolator
$\fortin : \vV \to \vV_h$ such that for all $\vv \in \vV$ and for all $q_h \in Q_h$ holds
\begin{equation}
  b(\fortin \vv, q_h) = \b(\vv, q_h)  \qquad \text{and} \qquad \norm{\nabla \fortin \vv}{L^2(\Omega)} \leq C_F \norm{\nabla \vv}{L^2(\Omega)},
\end{equation}
where $C_F$ denotes the stability constant of the Fortin interpolator \cite{girault:raviart, brezzifortin}.
Introducing the space of discrete divergence-free velocity functions
\begin{equation}
   \vV^0_h := \{ \vv_h \in \vV_h : b(\vv_h, q_h) = 0 \text{ for all }  q_h \in Q_h \},
\end{equation}
the following lemma is a classical result by the theory of mixed finite element methods \cite{girault:raviart, brezzifortin}.
\begin{lemma} \label{lemma:V0h:2:Vh}
Let the finite element spaces $\vV_h$ and $Q_h$ fulfill the discrete LBB condition \eqref{discrete:LBB:cond}, then it holds
for all $\vv \in \vV^0$
$$
  \inf_{\vv_h \in \vV^0_h} \norm{\nabla \vv - \nabla \vv_h}{L^2(\Omega)} \leq (1+C_F) \inf_{\vw_h \in \vV_h} \norm{\nabla \vv - \nabla \vw_h}{L^2(\Omega)}.
$$
\end{lemma}

In the following we  propose a non-standard discretization of the right hand side of the Stokes equations, in order to obtain
pressure-robust velocity error estimates.
Key is the definition of a velocity reconstruction operator in the spirit of \cite{Linke:2012, linke:cmame:2014}
that maps discrete divergence-free velocity test functions to exact divergence-free ones.
The novelty of this contribution is that we define such reconstruction operators for mixed finite element methods,
which possess only {\em continuous} discrete pressures.
The most prominent examples of such mixed finite element methods are given by the Taylor--Hood element family
and the mini element \cite{girault:raviart, brezzifortin}. From now on we focus on the Taylor--Hood element of order $k \geq 2$ so
\begin{align*}
  \Vh := [\pol^k(\Tc)]^d \cap [C^0(\Omega)]^d \quad \text{and} \quad \Qh := \pol^{k-1}(\Tc) \cap C^0(\Omega),
\end{align*}
and give a detailed description for the mini element in Section \ref{section::mini}. 
The velocity reconstruction operators 
\begin{align*}
  \Rec: \vV_h \rightarrow \vV_h + \Sigh
\end{align*}
%by means of solving small problems on vertex patches. A precise definition is given in section \ref{sec::ana}.
%\begin{theorem}\label{theoremone} The reconstruction operator $\Rec$ defined by equation \ref{defreconstruction} satisfies
with some $H(\mathrm{div})$-conforming finite element space $\Sigh$
are defined  by solving local problems on vertex patches.
A precise definition is given in Section \ref{sec::ana}. 
We introduce the discrete space of scalar functions
\begin{equation}
  \tQh := \div (\Rec \Vh),
\end{equation}
and we assume that it holds $Q_h \subset \tQh$. The Oswald interpolator is now defined from $\oswald: \tQh \rightarrow \Qh$ with the property
\begin{align} \label{oswaldid}
  \restr{\oswald}{\Qh(\Omega)} = \textrm{id}.
\end{align}
For the error estimates of the finite element method to be proposed, we use the following abstract properties of $\Rec$, which are summarized 
in the following theorem.
\begin{theorem}\label{theoremone} For the reconstruction operator $\Rec$ defined by equation \eqref{defreconstruction} holds
  \begin{align}
    & i.   && \skp{\div{\Rec \vw_h}, \tqh} = \skp{\div{ \vw_h}, \oswald \tqh}  \quad \forall \tqh \in \tQh,  \label{rec:propBBB}\\
    & ii.  && \skp{\div{(\vw_h - \Rec \vw_h)}, q_h} = 0 \quad \forall \vw_h \in \vV_h, \forall q_h \in \Qh, \label{rec:propA}\\
    & iii. && \skp{\div \vw_h , q_h} = 0 ~ \forall q_h \in Q_h \Rightarrow \skp{\div{\Rec \vw_h}, \tqh} = 0 ~ \forall \tqh \in \tQh,\label{rec:propB}\\
    &      && \hspace{6cm} \text{ i.e.} \quad \div{\Rec \vw_h} = 0,\nonumber\\
    & iv.  && \skp{\vg, \vw_h - \Rec \vw_h} \le C_{\mathrm{cons}} \normsmooth{\vg}_{k-2} \|\nabla \vw_h\|_{L^2(\Omega)} \text{ for any } \vg \in [L^2(\Omega)]^d \label{rec:propC}
  \end{align}
  with {\it data oscillation} defined by \(\normsmooth{\vg}_{m} :=\left( \sum\limits_{V \in \Vert} h_V^2 \normlo{\vg - \mathcal{P}_{\ov}^m \vg}^2 \right)^{\frac{1}{2}}\).
%\item
  %If $\skp{\div \vw_h , q_h} = 0 ~ \forall q_h \in Q_h$ then 
  %\begin{align}
  %   \skp{\div{\Rec w_h}, \tqh} = 0 \quad \forall \tqh \in \tQh, \quad \text{i.e.} \quad \div{\Rec \vw_h} = 0. \label{rec:propB}
     %    \end{align}
 %\begin{align} \label{rec:propBB} 
  %  \norm{\div{\Rec \w_h}}{L^2(\Omega)} \le C_{\oswald} \norm{\div{\w_h}}{L^2(\Omega)} ,
  %\end{align}
\end{theorem}
\begin{remark} The {\it data oscillation} $\normsmooth{\cdot}_m$ is similar to a estimation used  for the analysis of adaptive methods, see for example \cite[p. 60]{aposteriori}.
%\todo[inline]{reference auf carstensen, Christian: ich wuerde den Remark weglassen}
Note that for $\vg \in H^{l}(\Omega)$ and a quasi--uniform  triangulation $\Tc$  it follows using a scaling argument that
\begin{align*}
  \normsmooth{\vg}_{m} \cle h^{\min\{m+2,l+1\}} |\vg|_{H^{l}(\Omega)}.
\end{align*}  
\end{remark}
The discrete Stokes problem can now be defined by: search for $(\vu_h, p_h) \in \Vh \times Q_h$ such that for all $(\vv_h, q_h) \in \Vh \times Q_h$ holds
\begin{equation}
\begin{split}
   a(\vu_h, \vv_h) + b(\vv_h, p_h) & = l(\Rec \vv_h), \label{disc:weak:stokes:problem} \\
     b(\vu_h, q_h) & = 0.
\end{split}
\end{equation}
\begin{remark}
The stiffness matrix of the proposed discretization \eqref{disc:weak:stokes:problem} is the same as for standard inf-sup stable mixed finite element methods.
However, the discretization of the right hand side is non-standard.
The main reason for this non-standard discretization is: for the continuous Stokes problem \eqref{cont:weak:stokes:problem} it holds that
$(\vu, \psi)$ is the solution for arbitrary right hand sides of the form
$\bfe = \nabla \psi$ with $\psi \in H^1(\Omega) \big/ \R$, i.e., irrotational forces $\bfe = \nabla \psi$ lead to a no-flow velocity solution
$\vu = \vzero$ \cite{Linke:2012, linke:cmame:2014}. This is due to the $L^2$ orthogonality $\intO \nabla \psi \cdot \vw \intd x = 0$
for all $\vw \in \Hdivoz{\Omega}$ with $\div \vw = 0$.
Similarly it holds $\vu_h = \vzero$ for the discretization \eqref{disc:weak:stokes:problem}, since due to Theorem \ref{theoremone}
discrete divergence-free velocity test functions are mapped to divergence-free ones \cite{Linke:2012, linke:cmame:2014}.
\end{remark}
%%% Local Variables: 
%%% mode:latex
%%% TeX-master: "siam-reconstruction"
%%% End: 

%% file: rec_errorest.tex
% SIAM Shared Information Template
% This is information that is shared between the main document and any
% supplement. If no supplement is required, then this information can
% be included directly in the main document.

In this section, an a-priori error analysis is performed for the solution of the discrete Stokes problem $(\vu_h, p_h)$ in \eqref{disc:weak:stokes:problem}.
The following lemma is needed to estimate the consistency error introduced due to the non-standard discretization of the right hand side in \eqref{disc:weak:stokes:problem}.
\begin{lemma} \label{lemma:est:consist:error}
For $\vv \in \vV$ with $\Delta \vv \in [L^2(\Omega)]^d$ and for all $\vw_h \in \Vh$ it holds
$$
  \lvert (\Delta \vv, \Rec \vw_h) + (\nabla \vv, \nabla \vw_h) \rvert \leq C_\mathrm{cons}   \normsmooth{\Delta \vv}_{k-2} \|\nabla \vw_h\|_{L^2(\Omega)}.
$$
\end{lemma}
\begin{proof}
By calculating and applying \eqref{rec:propC}, one obtains
\begin{align*}
  (\Delta \vv, \Rec \vw_h) + (\nabla \vv, \nabla \vw_h) & = (\Delta \vv, \Rec \vw_h - \vw_h) + (\Delta \vv, \vw_h) + (\nabla \vv, \nabla \vw_h) \\
  & = (\Delta \vv, \Rec \vw_h - \vw_h) \leq  C_\mathrm{cons}  \normsmooth{\Delta \vv}_{k-2} \norm{\nabla \vw_h}{L^2(\Omega)}.
\end{align*}
\end{proof}

\begin{theorem} \label{Theorem:apriori:errors}
For the discrete solution $(\vu_h, p_h) \in \Vh \times Q_h$ in \eqref{disc:weak:stokes:problem} and the continuous solution $(\vu, p) \in (\vV, Q)$ of \eqref{cont:weak:stokes:problem},
assuming the regularity $\Delta \vu \in [L^2(\Omega)]^d$ %and $p \in H^1(\Omega)$,
the following a-priori errors hold
\begin{align}
 & i. && \norm{\nabla (\vu - \vu_h)}{L^2(\Omega)} \leq 2 (1+ C_F) \inf_{\vw_h \in \vV_h} \norm{\nabla (\vu - \vw_h)}{L^2(\Omega)} + C_\mathrm{cons} \normsmooth{\Delta \vu}_{k-2},\nonumber\\
 & ii. &&  \|\oswald\mathcal{P}_{\tQh} p - p_h\|_{L^2(\Omega)} \leq
     \frac{\nu}{\beta_h} \left ( \norm{\nabla (\vu - \vu_h)}{L^2(\Omega)} +  C_{\mathrm{cons}} \normsmooth{\Delta \vu}_{k-2} \right ),\label{bestapproxpressure_estimate}\\
 & iii. && \norm{p - p_h}{L^2(\Omega)} \leq \|p - \oswald \mathcal{P}_{\tQh} p\|_{L^2(\Omega)}\nonumber\\
 &      && \hspace{4cm}  + \frac{\nu}{\beta_h} \left ( \norm{\nabla (\vu - \vu_h)}{L^2(\Omega)} +  C_\mathrm{cons} \normsmooth{\Delta \vu}_{k-2} \right ).\nonumber
\end{align}
\end{theorem}
\begin{proof}
Note that from $\Delta \vu \in [L^2(\Omega)]^d$ and $\bfe \in [L^2]^d(\Omega)$ follows $p \in H^1(\Omega)$. 
i) For an arbitrary $\vv_h \in \vV^0_h$ we define $\vw_h := \vu_h - \vv_h \in \vV^0_h$.
\begin{align*}
  \nu \norm{\nabla \vw_h}{L^2(\Omega)}^2
  & = a(\vw_h, \vw_h) = a(\vu_h, \vw_h) - a(\vv_h, \vw_h)\\
  & = (-\nu \Delta \vu + \nabla p, \Rec \vw_h)  -  a(\vv_h, \vw_h)  \\
  & = a(\vu - \vv_h, \vw_h) - \nu
  \left( (\Delta \vu, \Rec \vw_h) + (\nabla \vu, \nabla \vw_h) \right),
\end{align*}
where it was used that $\div \Rec \vw_h = 0$ holds due to \eqref{rec:propB} and that thus
$\nabla p$ and $\Rec \vw_h$ are orthogonal in $L^2$.
Using Lemma \ref{lemma:est:consist:error} and the Cauchy--Schwarz inequality yields
\begin{equation*}
     \nu \norm{\nabla \vw_h}{L^2(\Omega)}^2
  \leq \nu \norm{\nabla(\vu - \vv_h)}{L^2(\Omega)} \norm{\nabla \vw_h}{L^2(\Omega)}
  + \nu C_\mathrm{cons}  \normsmooth{\Delta \vu}_{k-2} \norm{\nabla \vw_h}{L^2(\Omega)}.
\end{equation*}
Therefore it holds
$$
  \norm{\nabla \vw_h}{L^2(\Omega)} \leq \inf_{\vv_h \in \vV^0_h} \norm{\nabla(\vu - \vv_h)}{L^2(\Omega)} + C_\mathrm{cons} \normsmooth{\Delta \vu}_{k-2}.
$$
With the triangle inequality it follows
\begin{equation*}
  \norm{\nabla (\vu - \vu_h)}{L^2(\Omega)}
  \leq \norm{\nabla (\vu - \vv_h)}{L^2(\Omega)}
  +  \norm{\nabla \vw_h}{L^2(\Omega)}.
\end{equation*}
Applying Lemma \ref{lemma:V0h:2:Vh} yields the first statement.

ii) For proving the pressure error, one computes for an arbitrary $\vv_h \in \vV_h$
%Using now \eqref{WeakStatement_mod1}, integration by parts, and \eqref{eq:thm_5_2} yields
\begin{equation*}
\begin{split}
    (\oswald\mathcal{P}_{\tQh} p - p_h, \div \vv_h)
  & = (\oswald\mathcal{P}_{\tQh} p, \div \vv_h) + (\bfe, \Rec \vv_h)
  - a(\vu_h, \vv_h) \\
  & = (\mathcal{P}_{\tQh} p, \div \Rec \vv_h) +  (\nabla p, \Rec \vv_h)
  -(\nu \Delta \vu, \Rec \vv_h)
  - a(\vu_h, \vv_h) \\
  & =  -(\nu \Delta \vu, \Rec \vv_h) - a(\vu_h, \vv_h) \\
  & =  -(\nu \Delta \vu, \Rec \vv_h) - a(\vu, \vv_h) - a(\vu_h - \vu, \vv_h),
\end{split}
\end{equation*}
where \eqref{rec:propBBB} was used.
Using the discrete LBB condition \eqref{discrete:LBB:cond}, one concludes
$$
\|\oswald\mathcal{P}_{\tQh} p - p_h\|_{L^2(\Omega)} \leq
     \frac{\nu}{\beta_h} \left ( \norm{\nabla (\vu - \vu_h)}{L^2(\Omega)} +  C_{\mathrm{cons}} \normsmooth{\Delta \vu}_{k-2} \right ).
$$

iii)
The last statement follows by the triangle inequality.
\end{proof}
\begin{remark}
The statement i) in Theorem \ref{Theorem:apriori:errors} shows the pressure-robustness of the a-priori velocity error.
Interesting is also statement ii) in Theorem \ref{Theorem:apriori:errors}. It shows that also the pressure error is pressure-robust
in the sense that $p_h = \oswald \mathcal{P}_{\tQh}p$ up to an error, which is only velocity-dependent.
Note that this is completely analogous to pressure-robust mixed methods with discontinuous pressures \cite{lmt:2016, Linke21052016, blms:2015}. There, $Q_h$ and $\tQh$ coincide
and $p_h$ is even the {\em best approximation} of $p$ in $Q_h$ up to an error, which is also only velocity-dependent.
\end{remark}
\begin{corollary} \label{Corollary:th:orders}
  Assume a quasi-uniform triangulation $\Tc$ and a solution $\vu \in [H^{k+1}(\Omega)]^d$ and $p \in H^{k}(\Omega)$ of the continuous problem \eqref{cont:weak:stokes:problem}.
  Then, the solution $(\vu_h, p_h)$ of \eqref{disc:weak:stokes:problem} satisfies
  \begin{align} 
    \| \vu - \vu_h \|_{H^1(\Omega)} & \le \left(2(1+C_F) + C_\mathrm{cons}\right) h^{k} |\vu|_{H^{k+1}(\Omega)}, \quad \text{and}\label{vel:order:est}\\
    \| p - p_h \|_{L^2(\Omega)} & \le \frac{\nu\left(2(1+C_F) + 2C_\mathrm{cons}\right)}{\beta} h^{k} |\vu|_{H^{k+1}(\Omega)} + h^{k} |p|_{H^{k}(\Omega)}. \label{pres:order:est}
  \end{align}
\end{corollary}
\begin{proof}
Follows by Theorem \ref{bestapproxpressure_estimate} and standard scaling arguments.
\end{proof}
\begin{remark}
In order to increase the accuracy of the solution one may want to use a local refinement of the mesh $\Tc$. This is indeed possible with the modified method due to local properties of the {\it data oscillation}.
\end{remark}
%For the $L^2$ estimate we introduce the dual problem for arbitrary $\vg \in [L^2(\Omega)]^d$:
%search for $(\vw, \theta) \in \vV \times Q$ such that for all $(\vv, q) \in \vV \times Q$ holds
%\begin{align}
%\begin{split}
%   a(\vw, \vv) + b(\vv, \theta) & = (\vg, \vv)_{L^2(\Omega)} \label{cont:weak:stokes:dualproblem} \\
%     b(\vw, q) & = 0.
%\end{split}
%\end{align}
\begin{corollary}
  Under the assumptions of Theorem \ref{Theorem:apriori:errors}, Corollary \ref{Corollary:th:orders} and the convexity of $\Omega$ it holds
  \begin{align*}
    \| \vu - \vu_h \|_{L^2(\Omega)} \cle  h^{k+1} |\vu|_{H^{k+1}(\Omega)}.
  \end{align*}
\end{corollary}
\begin{proof}
The proof follows by an Aubin--Nitsche argument \cite{brezzifortin, aubin1972, NITSCHE1968}.
For an arbitrary $\vg \in [L^2(\Omega)]^d$ one employs a dual Stokes problem with a solution $\vu_{\vg} \in \vV^0 \cap [H^2(\Omega)]^d$.
Extending the domain of definition of the reconstruction operator $\Rec$ to $\vV^0$ one sees at once that it holds $\Rec \vw = \vw$ for all $\vw \in \vV^0$.   
Then, $\Rec \vu_{\vg} = \vu_{\vg}$ and the arguments in \cite{lmt:2016} deliver the desired optimal pressure-robust $L^2$-estimate.
\end{proof}
%%% Local Variables: 
%%% mode:latex
%%% TeX-master: "siam-reconstruction"
%%% End: 

%% file: rec_analysis.tex
% SIAM Shared Information Template
% This is information that is shared between the main document and any
% supplement. If no supplement is required, then this information can
% be included directly in the main document.

\subsection{Definition of the operators and spaces}
In this section we define local problems on each vertex patch $\ov$ and proof theorem \ref{theoremone}. For an arbitrary vertex $V \in \Vert$ we start by defining the spaces
\begin{align*}
  \Sigho (\Tc_{\ov}) &:= \{ \sig \in \RT{k-1}(\Tc_{\ov}): \tracen \sig = 0 \ondom \partial \ov  \} \subset \Hdivoz{\ov} \\
  \tQh (\Tc_{\ov}) &:= \pol^{k-1}(\Tc_{\ov}) \subset L^2(\ov) \qquad   \tQh^0 (\Tc_{\ov}) := \tQh (\Tc_{\ov}) \cap L^2_0(\ov),
\end{align*}
where $\RT{k-1}$ is the Raviart-Thomas space of order $k-1$ see \cite{brezzifortin} and \cite{RTelements} , and for $k \ge 3$ using the Koszul operator also
\begin{align*}
   \Wh (\ov) &:= \kos_{\vec{x} - V}(\pol^{k-3}(\ov)) \subset \Lambda_V :=  \kos_{\vec{x} - V}(L^2(\ov)) \quad \text{for } \quad d=2 \\
  \Wh (\ov) &:= \kos_{\vec{x} - V}([\pol^{k-3}(\ov)]^3)  \subset  \Lambda_V := \kos_{\vec{x} - V}([L^2(\ov)]^3) \quad \text{for } \quad d=3.
\end{align*}
Note that $\tQh$ consists of element-wise polynomials and $\pol^{k-3}(\ov)$ are polynomials on the patch. Furthermore we have the property
\begin{align} \label{divSigeqtQ}
 \div{\Sigho (\Tc_{\ov})} = \tQh^0 (\Tc_{\ov}). %\cap L^2_0(\ov) \simeq \tQh (\Tc_{\ov}) / \R.
\end{align}
We continue with the definition of the bilinearform $\blfB: (\Hdivoz{\ov} \times L^2_0(\ov) \times \Lambda_V) \times (\Hdivoz{\ov} \times L^2_0(\ov) \times \Lambda_V) \rightarrow \R$ by
\begin{multline*}
  \blfB((\sigb, \phi, \boldsymbol{\lambda}), (\boldsymbol{\tau}, \psi, \boldsymbol{\mu})) := \\
  \into \sigb \cdot \boldsymbol{\tau} \intd x + \into \div{\boldsymbol{\tau}} \phi \intd x +  \into \boldsymbol{\tau} \cdot \boldsymbol{\lambda} \intd x
  + \into \div{\sigb} \psi \intd x + \into \sigb \cdot \boldsymbol{\mu} \intd x.
\end{multline*}
Now let $T$ be an arbitrary element $T \in \Tc$, and $\Vert_T$  be the set of vertices of $T$ with $N_T := |\Vert_T|$.  Let $\{\phi_j\}_{j=1}^{N_T}$ be the local (Lagrangian) basis on $T$ for the interpolation points $\{x_j\}_{j=1}^{N_T}$ and $\{q_j\}_{j=1}^{N_T}$ be the coefficients of an arbitrary $q \in \pol^{k-1}(T)$, so 
  \begin{align*}
    \phi_j(x_l) = \delta_{jl} \quad \forall j,l=1,\dots,N_T \qquad \andtext \qquad q(x) = \sum\limits_{j=1}^{N_T} q_j \phi_j(x).
  \end{align*}
  Then we define for each $V \in \Vert_T$ an operator $\bubb_{T,V} : \pol^{k-1}(T) \rightarrow \pol^{k-1}(T)$ by setting the coefficients as
  \begin{align} \label{coeffbubb}
     (\bubb_{T, V} q)_j=  q_j \lambda_{V}(x_j),
  \end{align}
  where $\lambda_{V}$ is the barycentric coordinate function of the vertex $V$. Figure~\ref{fig:vis_bubbleprojector} visualizes the change in the
  coefficients for a quadratic polynomial in two dimensions.
  It holds
\begin{align} \label{bubbprojprop}
  \trace \bubb_{T,V} q = 0 ~ \ondom ~  F_{op} \qquad  \andtext  \qquad  \sum\limits_{V \in \Vert_T} \bubb_{T,V} q = q,
\end{align}
where $F_{op}$ is the opposite edge of $V$ for $d = 2$ and the opposite face for $d=3$. Using a trivial extension by $0$ on $\Omega \setminus T$, we can expand the range of $\bubb_{T,V}$ on $\tQh(\Tc)$. By that we define for every vertex $V$ the {\it bubble projector}  $\bubb_V: \tQh(\Tc) \rightarrow \tQh(\Tc)$ as
\begin{align} \label{bubbproj}
  \bubb_V \tqh := \sum\limits_{T \in \TcV} \bubb_{T,V} \tqh \quad \forall \tqh \in \tQh(\Tc),
\end{align}
with the property
\begin{align} \label{bubbprojB}
  \trace \bubb_V \tqh & = 0 \quad \ondom \quad \partial \ov. \\
  \bubb_V \tqh &= 0 \quad \ondom \quad \Omega \setminus \ov.
\end{align}

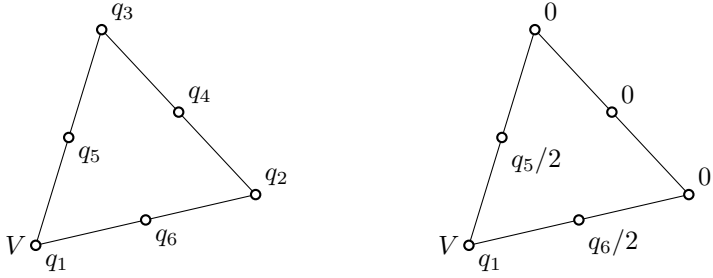
\begin{figure}

\tikzstyle{point}=[
inner sep=1.25pt,
fill=white,
draw,
circle,
thick
]
\hfill
 \begin{tikzpicture}[scale=1.5, rotate=13]
\path (0,0) coordinate (P1)
      (2,0) coordinate (P2)
      (1,1.732) coordinate (P3)
      (P1)--(P2) coordinate[midway] (E3)
      (P2)--(P3) coordinate[midway] (E1)
      (P3)--(P1) coordinate[midway] (E2);
\draw (P1)--(P2)--(P3)--(P1);

\foreach \x in {P1,P2,P3,E1,E2,E3}
{%
\path (\x) coordinate[point];
}
\node at (P1)[left]{$V$};
\node at (P1)[below right]{$q_1$};
\node at (P2)[above right]{$q_2$};
\node at (P3)[above right]{$q_3$};
\node at (E1)[above right]{$q_4$};
\node at (E2)[below right]{$q_5$};
\node at (E3)[below right]{$q_6$};
 \end{tikzpicture}
 \hfill
  \begin{tikzpicture}[scale=1.5, rotate=13]
\path (0,0) coordinate (P1)
      (2,0) coordinate (P2)
      (1,1.732) coordinate (P3)
      (P1)--(P2) coordinate[midway] (E3)
      (P2)--(P3) coordinate[midway] (E1)
      (P3)--(P1) coordinate[midway] (E2);
\draw (P1)--(P2)--(P3)--(P1);

\foreach \x in {P1,P2,P3,E1,E2,E3}
{%
\path (\x) coordinate[point];
}
\node at (P1)[left]{$V$};
\node at (P1)[below right]{$q_1$};
\node at (P2)[above right]{$0$};
\node at (P3)[above right]{$0$};
\node at (E1)[above right]{$0$};
\node at (E2)[below right]{$q_5/2$};
\node at (E3)[below right]{$q_6/2$};
 \end{tikzpicture}
 \hfill\hfill
\caption{Visualisation of the nodal coefficients \(q_1,\ldots,q_6\) of a quadratic polynomial $q \in \pol^{2}(T)$ (left)
and the coefficients of its bubble projector $\bubb_{T,V} q$ (right) on a triangle $T$ with respect to the
vertex $V$.\label{fig:vis_bubbleprojector}}
\end{figure}

\begin{figure}[h]
  \centering
  \subfloat[Arbitrary polynomial function $\tqh$]
  {\includegraphics[width=0.25\paperwidth, clip=true, trim=5cm 5cm 8cm 1cm]
    {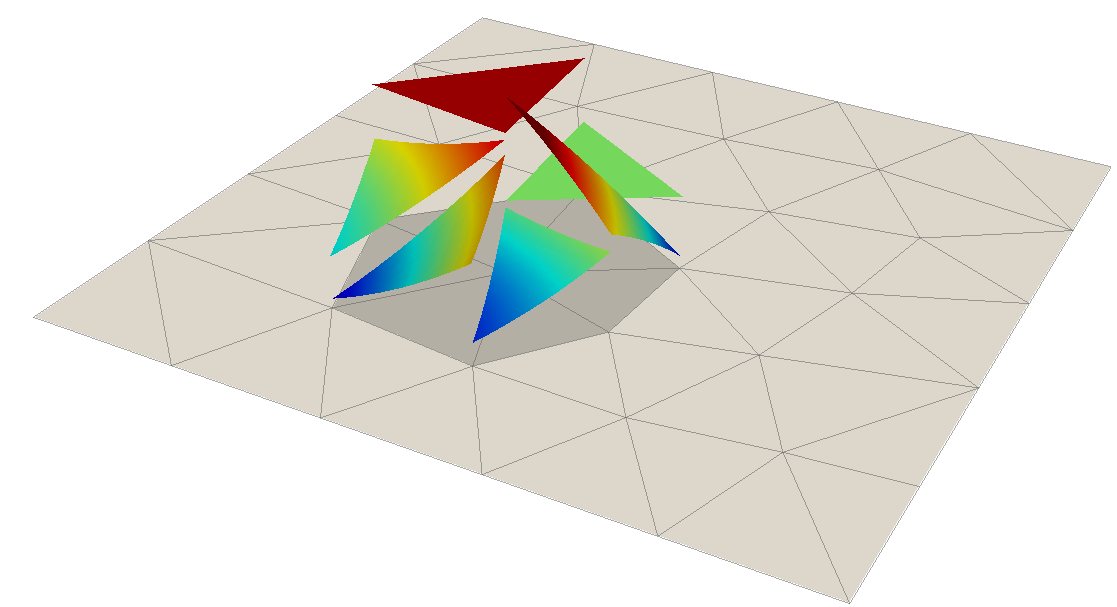}}
  \hspace*{10pt}
  \subfloat[Applying the bubble Projector $\bubb_V \tqh$]
  {\includegraphics[width=0.25\paperwidth , clip=true, trim=5.1cm 4.9cm 8cm 0.9cm]
    {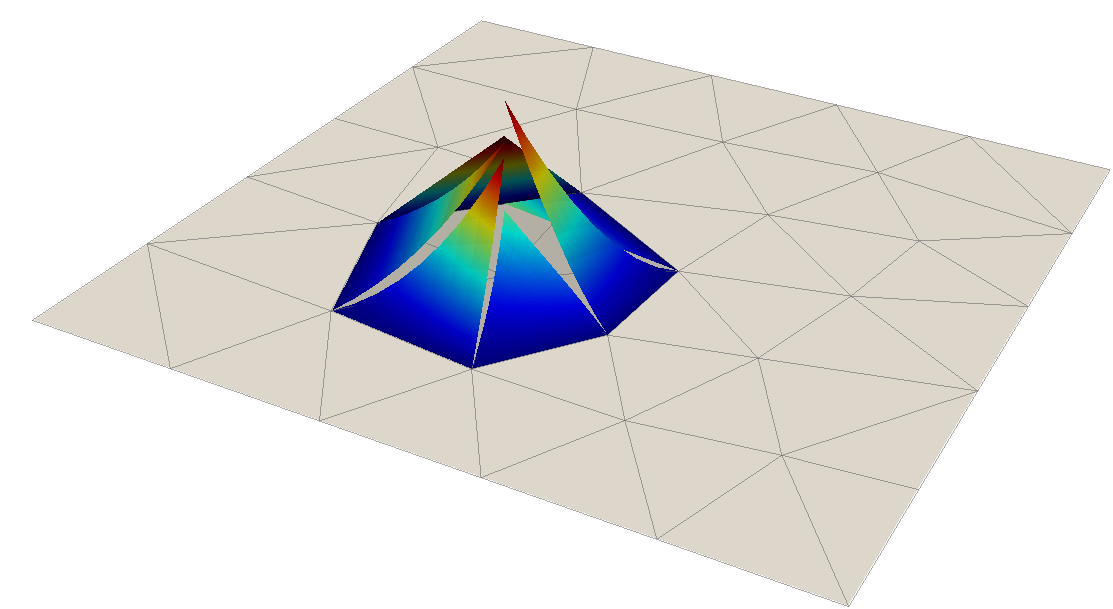}}
  \caption{An example for the bubble projector on $\ov$ (dark gray)}\label{pic::bubbproj}
\end{figure}
In Figure \ref{pic::bubbproj} an example of a projected arbitrary $\tqh \in \tQh$ is given.

\begin{remark}
  More complicated, but polynomial-robust {\it bubble projectors} are given in \cite{precondpversion} and \cite{bubbletransform}. If this robustness is an issue, these operators could be used instead of $\bubb_V$.
  %For an implementation of the {\it bubble projector}  with a different basis we refer to the operators constructed in \cite{precondpversion} and \cite{bubbletransform}.
\end{remark}
\subsection{Definition of the local problem}
On the vertex patch, we define the problem: For a given function $\div{\w_h} \in \tQh(\Tc_{\ov})$ find $(\sigV, \phi_h, \boldsymbol{\lambda}_h) \in (\Sigho (\Tc_{\ov}) \times \tQh^0 (\Tc_{\ov}) \times \Wh(\ov))$ so that
\begin{align}\label{patchproblem}
  \blfB((\sigV, \phi_h, \boldsymbol{\lambda}_h), (\boldsymbol{\tau}_h, \psi_h, &\boldsymbol{\mu}_h)) = \skpV{\div{\w_h}, \bubb_V \left( \psi_h - \oswald \psi_h \right)} \\
                                                                        &\forall (\boldsymbol{\tau}_h, \psi_h, \boldsymbol{\mu}_h) \in \Sigho (\Tc_{\ov}) \times \tQh^0 (\Tc_{\ov}) \times \Wh(\ov). \nonumber
\end{align}
\begin{theorem} \label{localtheorem}
  Equation \ref{patchproblem} has a unique solution $(\sigV, \phi_h, \boldsymbol{\lambda}_h)$ satisfying
  \begin{align}
    & i.  && \normlo{\sigV} \cle h_V \normlo{\div{\w_h}}, \label{lemmaonepropone}\\
    & ii. && \skp{\div{\sigV}, \tqh} = \skpV{\div{\w_h}, \bubb_V \left( \tqh - \oswald \tqh \right)} \quad \forall \tqh \in \tQh(\Tc) \label{lemmaoneproptwo}\\
    &     && \text{ where $\sigV$ was trivially extended by 0 on $\Omega$,}\nonumber\\
    & iii. && \text{and the solution is $L^2(\ov)$-orthogonal to polynomials of order $k-2$, i.e. }\nonumber\\
    &      && \skpV{\sigV, \xi} = 0 \quad \forall \xi \in [\pol^{k-2}(\ov)]^d.\label{lemmaonepropthree}
  \end{align}
\end{theorem}
\begin{proof}[Proof of existence, uniqueness and i]
We start with the considered norms
  \begin{align*}
    \| \boldsymbol{\tau}_h \|_{\Sigho (\Tc_{\ov})} &:= \normlo{\boldsymbol{\tau}_h} + h_V \normlo{\div{\boldsymbol{\tau}_h}},\\
    \| \psi_h \|_{\tQh (\Tc_{\ov})} &:= \frac{1}{h_V} \normlo{\psi_h},\\
                                      \| \boldsymbol{\mu}_h \|_{\Wh(\ov)} &:= \normlo{\boldsymbol{\mu}_h}.
  \end{align*}
%  and for $\mu_h= \newvec{y-V_y}{-x+V_x} a_h$ for $d=2$ and $\mu_h= \newvec{y-V_y}{-x+V_x} \times a_h$ for $d=3$ the norm 
%  \begin{align*}
%        \| \mu_h \|_{\Wh(\ov)} &:= \normlo{\a_h}.
%  \end{align*}
  In this part of the proof we use $\Sigho$ as symbol for $\Sigho(\Tc_{\ov})$ and similar for $\tQh(\Tc_{\ov})$ and $\Wh(\Tc_{\ov})$.
  Next we define the bilinearforms
  \begin{align*}
    a_\sigma(\sig,\boldsymbol{\tau}_h) &:=\int_{\ov} \sig \cdot \boldsymbol{\tau}_h \intd x \qquad \forall (\sig,\boldsymbol{\tau}_h) \in \Sigho \times \Sigho,\\
    b_1(\sig,\psi_h) &:=\displaystyle\int_{\ov} \div{\sig} \psi_h \intd x  \qquad \forall (\sig,\psi_h) \in \Sigho  \times \tQh, \\
    b_2(\sig,\boldsymbol{\mu}_h) &:=\displaystyle\int_{\ov} \sig \cdot \boldsymbol{\mu}_h \intd x \qquad \forall (\sig,\boldsymbol{\mu}_h) \in \Sigho \times \Wh.
  \end{align*}
  Using the Cauchy Schwarz inequality we see that $a_\sigma, b_1$ and $b_2$ are all continuous
  \begin{align*}
    a_\sigma(\sig,\boldsymbol{\tau}_h) \cle  \normlo{\sig} \normlo{\boldsymbol{\tau}_h} \cle  \normsigo{\sig} \normsigo{\boldsymbol{\tau}_h} \\
    b_1(\sig,\psi_h) \cle  \normlo{\div{\sig}} \normlo{\psi_h} \cle  \normsigo{\sig} \normtqo{\psi_h} \\
    b_2(\sig,\boldsymbol{\mu}_h) \cle \normlo{\sig}\normlo{\boldsymbol{\mu}_h} = \normwo{\sig}\normwo{\boldsymbol{\mu}_h}.
  \end{align*}
  As \begin{align*}
       \blfB((\sigV, \phi_h, \boldsymbol{\lambda}_h), &(\boldsymbol{\tau}_h, \psi_h, \boldsymbol{\mu}_h)) = \\
       &a_\sigma(\sig,\boldsymbol{\tau}_h) + b_1(\sig,\psi_h) + b_2(\sig,\boldsymbol{\mu}_h)  + b_1(\boldsymbol{\tau}_h,\phi_h) + b_2(\boldsymbol{\tau}_h,\boldsymbol{\lambda}_h),
     \end{align*}
     we show the existence and uniqueness of the saddle point problem (\ref{patchproblem}) as in chapter 4 in \cite{brezzifortin}, so it remains to show the ellipticity of $a_\sigma(\cdot,\cdot)$, i.e. 
     \begin{align} \label{kernelcoerc}
       a_\sigma(\sig,\sig) \cge &\normsigo{\sig}^2 \quad \forall \sig \in   \Sigho^0      
     \end{align}
      on the kernel
     \begin{align*}
       \Sigho^0 := \{ \sig \in \Sigho: b_1(\sig,\psi_h) + b_2(\sig,\boldsymbol{\mu}_h) = 0 ~ \forall (\psi_h, \boldsymbol{\mu}_h) \in \tQh^0 \times \Wh\},
     \end{align*}
     and the LBB condition with some $\beta_\sigma > 0$ such that, for all $(\psi_h,\boldsymbol{\mu}_h) \in \tQh^0 \times \Wh$,
     \begin{align} \label{LBBcondition}
       \sup_{\sig \in \Sigho} \frac{b_1(\sig,\psi_h) + b_2(\sig,\boldsymbol{\mu}_h) }{\normsigo{\sig}} \cge \beta_\sigma (\normtqo{\psi_h} + \normwo{\boldsymbol{\mu}_h}).
     \end{align}
     For a function $\sig$ in the kernel $\Sigho^0$ it holds in particular     
     \begin{align*}
       \b_1(\sig, \psi_h) = 0 \quad \forall \psi_h \in \tQh^0,
     \end{align*}
     and hence $\div{\sig}=0$, thus
     \begin{align*}
      \normlo{\sig} = \normsigo{\sig} \quad  \forall \sig \in   \Sigho^0. 
     \end{align*}
     This implies \eqref{kernelcoerc}. To show \eqref{LBBcondition} we will proceed in three steps. First we show the LBB condition for the bilinearform $b_1(\cdot, \cdot)$ and then for $b_2(\cdot, \cdot)$ by choosing proper candidates that do not destroy the first condition, and finally combine the two estimates. For $b_1(\cdot,\cdot)$ we first show the LBB condition on the reference patch $\widehat{\omega_V}$ and then on $\omega_V$. It should be mentioned that there exist different reference patches due to the number of elements that belong to a vertex, but for each triangulation $\Tc$ there exist a finite number of reference patches. We use the standard Raviart-Thomas interpolator $\IRT$ of order $k-1$ (see \cite{brezzifortin}, or \cite{RTinterpolator}) that provides 
  \begin{align*}
    b_1(\IRT \sigb,\psi_h) = b_1(\sigb,\psi_h) \quad \forall \psi_h \in \tQh(\widehat{\omega_V}) \quad \forall \sigb \in \Hdiv{\widehat{\ov}}
  \end{align*}
and
  \begin{align*}
    \|\IRT \sigb\|_{\Hdiv{\widehat{\omega_V}}} \cle \norm{\sigb}{H^1(\widehat{\omega_V})} \quad \forall \sigb \in [H^1(\widehat{\omega_V})]^d.
  \end{align*}
For an arbitrary  $\hat{\psi_h} \in \tQh^0(\widehat{\omega_V})$ we have
    \begin{align*}
      \sup\limits_{\hat{\sigb}_h \in \Sigma_{h,0}(\widehat{\omega_V})} \frac{b_1(\hat{\sigb}_h,\hat{\psi}_h)}{\norm{\hat{\sigb}_h}{\Hdiv{\widehat{\omega_V}}}} 
&\cge \sup\limits_{\hat{\sigb} \in [H_0^1(\widehat{\omega_V})]^d} \frac{b_1(\IRT \hat{\sigb},\hat{\psi}_h)}{\norm{\IRT \hat{\sigb}}{\Hdiv{\widehat{\omega_V}}}} \\
&\cge \sup\limits_{\hat{\sigb}\in [H_0^1(\widehat{\omega_V})]^d} \frac{b_1(\hat{\sigb},\hat{\psi}_h)}{\norm{\hat{\sigb}}{H^1(\widehat{\omega_V})}}.
    \end{align*}
Next we use the continuous Stokes LBB condition \eqref{contstokesLBB} to get 
\begin{align} \label{lbbrefpatch}
  \sup\limits_{\hat{\sigb}_h \in \Sigma_{h,0}(\widehat{\omega_V})} \frac{b_1(\hat{\sigb}_h,\hat{\psi}_h)}{\norm{\hat{\sigb}_h}{\Hdiv{\widehat{\omega_V}}}} 
& \ge \beta_1 \|\hat{\psi}_h\|_{L^2(\widehat{\omega_V})},
\end{align}
with $\beta_1>0$ that depends only of the shape and size of the triangles on the reference patch. To show the condition on $\omega_V$ we recall the definition of the Piola transformation. Let $F: \hat{T} \rightarrow T$ be the mapping of the reference triangle to an arbitrary element $T$. Then the Piola transformation is defined as
\begin{align*}
    \Pc(\hat{\sigb}) := \frac{1}{\det{F'}} F' \hat{\sigb} \quad \forall \hat{\sigb} \in [L^2(\Tref)]^d.
\end{align*}
For an arbitrary $\psi_h$ we now choose $\hat{\psi}_h =\psi_h$, and define $\sig^1 := \Pc (\hat{\sigb}_h)$ for  $\hat{\sigb}_h$ that delivers the supremum of Equation \eqref{lbbrefpatch}. Standard scaling arguments yield
\begin{align}\label{lbbbone}
\frac{b_1(\sig^1 , \psi_h)}{\norm{\sig^1}{\Sigma_{h,0}}} &=  \frac{\int_{\omega_V} \div{\sig^1}\psi_h \intd x}{ \norm{\sig^1}{L^2(\omega_V)} + h_V \norm{\div{\sig^1}}{L^2(\omega_V)}} 
                                                                 \cge  \frac{h_V^{(d-2)/2}\int_{\widehat{\omega_V}} \div{\hat{\sigb}_h}\hat{\psi}_h\intd x }{ \norm{\hat{\sigb}_h}{L^2(\widehat{\omega_V})} + \norm{\div{\hat{\sigb}_h}}{L^2(\widehat{\omega_V})}} \\
                                                               &\ge h_V^{(d-2)/2} \beta_1 \|\hat{\psi}_h\|_{L^2(\widehat{\omega_V})} = \beta_1\frac{1}{h_V} \norm{\psi_h}{L^2(\omega_V)} = \beta_1\norm{\psi_h}{\tQh}. \nonumber
\end{align}
We continue with the LBB condition for $b_2(\cdot, \cdot)$.
We start with the case $d=3$. Choose an arbitrary $\boldsymbol{\mu}_h=\kos_{\vec{x} - V}(\a_h)  \in W_h$ with $\a_h \in [\pol^{k-3}(\ov)]^3$. Furthermore, due to theorem \ref{app::theoremone},  we can assume that $\div{\a_h}=0$. Now we define
\begin{align*}
  \sig^2 := -\curl{(\lambda_V\a_h)}
\end{align*}
where $\lambda_V$ is the hat function of the vertex $V$. Note that we have
\begin{align} \label{bonekernelprop}
  b_1(\sig^2, \psi_h) = 0.
\end{align}
Using integration by parts we get
\begin{align*}
  b_2(\sig^2,\boldsymbol{\mu}_h) & =-\int_{\ov} \curl{(\lambda_V\a_h)} \cdot \kos_{\vec{x} - V}(\a_h) \intd x \\
                         &= -\int_{\ov} (\lambda_V\a_h) \cdot \curl{ \left((\vec{x} - V) \times \a_h \right)}  \intd x.
\end{align*}
Using basic vector calculus leads to
\begin{align*}
  \curl{ \left((\vec{x} - V) \times \a_h \right)}  &=  (\vec{x} - V)\underbrace{\div{\a_h}}_{=0} + \underbrace{\nabla(\vec{x} - V)}_{I} \a_h - \a_h \underbrace{\div{(\vec{x} - V)}}_{=3} - \nabla \a_h(\vec{x} - V) \\
  &=-2\a_h - \nabla \a_h(\vec{x} - V)
\end{align*}
and so
\begin{align*}
  b_2(\sig^2,\boldsymbol{\mu}_h) & =-\int_{\ov} (\lambda_V\a_h) \cdot (-2\a_h - \nabla \a_h(\vec{x} - V)) \intd x \\
                    &=\int_{\ov} 2 \lambda_V\a_h^2 \intd x +  \int_{\ov} \lambda_V\a_h \cdot \nabla \a_h(\vec{x} - V) \intd x \\
                    &=\int_{\ov} 2 \lambda_V\a_h^2 \intd x + \frac{1}{2} \int_{\ov} \lambda_V  \nabla \a_h^2 \cdot (\vec{x} - V) \intd x\\
                    %&=\int_{\ov} 2 \lambda_V\a_h^2 \intd x - \frac{1}{2} \int_{\ov} \a_h^2 \div{((\vec{x} - V)\lambda_V)} \intd x\\
                    &=\int_{\ov} 2 \lambda_V\a_h^2 \intd x - \frac{1}{2} \int_{\ov} \a_h^2 \underbrace{\div{((\vec{x} - V)\lambda_V)}}_{3 \lambda_V + \nabla \lambda_V (\vec{x} - V)} \intd x\\
                    &=\frac{1}{2} \int_{\ov} \lambda_V\a_h^2 \intd x - \frac{1}{2} \int_{\ov} \a_h^2 \nabla \lambda_V (\vec{x} - V) \intd x.
\end{align*}
On any $T \subset \TcV$ the gradient of $\lambda_V$ is equivalent to the scaled normal vector $\boldsymbol{n}_V$ on the face opposite to $V$, and one can see that
\(-\boldsymbol{n}_V \cdot (\vec{x} - V) \leq 0\), what finally leads to
     \begin{align} \label{lbbbtwo}
       \b_2(\sig^2,\boldsymbol{\mu}_h) \cge \beta_2 \normlo{\a_h}^2  \cge \beta_2 \normwo{\boldsymbol{\mu}_h}^2.
     \end{align}
For the case $d=2$ we proceed similar. For an arbitrary $\boldsymbol{\mu}_h= \kos_{\vec{x} - V}(\xi_h)  \in W_h$ with $\xi_h \in \pol^{k-3}(\ov)$ we define
\begin{align*}
\sig^2 := -\Curl{(\lambda_V\xi_h)}
\end{align*}
Again it holds property \eqref{bonekernelprop} and we see
     \begin{align*}
       \b_2(\sig^2,\boldsymbol{\mu}_h) &= -\int_{\ov}\Curl{(\lambda_V\xi_h)}  \cdot \kos_{\vec{x} - V}(\xi_h) \intd x \\
                          &= -\int_{\ov}\nabla{(\lambda_V\xi_h)} \cdot(\vec{x} - V) \xi_h\intd x \\
                          &= \int_{\ov}(\lambda_V\xi_h) \div{( (\vec{x} - V) \xi_h )}\intd x \\
                          &= \int_{\ov}2 \lambda_V\xi_h^2 +  \frac{1}{2} \int_{\ov} \lambda_V (\vec{x} - V) \nabla \xi_h^2 \intd x.       
     \end{align*}
     The rest is similar as before. Now we can show \eqref{LBBcondition}. For an arbitrary $\psi_h \in \tQh^0$ and $\boldsymbol{\mu}_h \in W_h$ we choose the functions $\sig^1, \sig^2$  that fulfill  Equations \eqref{lbbbone} and \eqref{lbbbtwo} and \eqref{bonekernelprop}. Furthermore we can scale $\sig^1$ and $\sig^2$ so that 
\begin{align*}
  \norm{\sig^1}{\Sigma_{h,0}} = \norm{\psi_h}{\tQh} \quad \text{and} \quad \norm{\sig^2}{\Sigma_{h,0}} = \normwo{\boldsymbol{\mu}_h}.
\end{align*}
For $\alpha = \frac{1}{\beta_1\beta_2}$ we define then $\sig = \sig^1 + \alpha \sig^2$ and get
     \begin{align*}
       b_1(\sig,\psi_h) + b_2(\sig,\boldsymbol{\mu}_h) &= b_1(\sig^1,\psi_h) + b_2(\sig^1,\boldsymbol{\mu}_h) + \alpha b_2(\sig^2,\boldsymbol{\mu}_h) \\
                                          &\cge \beta_1 \normtqo{\psi_h}^2 - \normsigo{\sig^1}  \normwo{\boldsymbol{\mu}_h} + \alpha \beta_2  \normwo{\boldsymbol{\mu}_h}^2 \\
                                          & \cge \beta_1 \normtqo{\psi_h}^2 - \normtqo{\psi_h} \normwo{\boldsymbol{\mu}_h} + \alpha \beta_2  \normwo{\boldsymbol{\mu}_h}^2.
     \end{align*}
     Using Young's inequality we have
     \begin{align*}
       \normtqo{\psi_h} \normwo{\boldsymbol{\mu}_h} \le \frac{\beta_1}{2} \normtqo{\psi_h}^2 + \frac{1}{2 \beta_1} \normwo{\boldsymbol{\mu}_h}^2,
     \end{align*}
     and so
     \begin{align*}
       b_1(\sig,\psi_h) + b_2(\sig,\boldsymbol{\mu}_h) &\cge \frac{\beta_1}{2} \normtqo{\psi_h}^2 + \frac{1}{2\beta_1} \normwo{\boldsymbol{\mu}_h}^2 \\
       &\cge \left(\frac{\beta_1}{2}  + \frac{1}{2\beta_1}\right) ( \normtqo{\psi_h}+\normwo{\boldsymbol{\mu}_h})^2.
     \end{align*}
     As $\normsigo{\sig} = \normsigo{\sig^1 + \alpha \sig^2} \le (1+\alpha) ( \normtqo{\psi_h}+\normwo{\boldsymbol{\mu}_h})$ we get
     \begin{align*}
       \frac{b_1(\sig,\psi_h) + b_2(\sig,\boldsymbol{\mu}_h) }{\normsigo{\sig}} \cge \beta (\normtqo{\psi_h} + \normwo{\boldsymbol{\mu}_h}) 
     \end{align*}
     and thus \eqref{LBBcondition} holds with $\beta_\sigma = \frac{\beta_1^2 + 1}{2\beta_1(1+\alpha)}$. Using the theory of saddle point problems, chapter 4 in \cite{brezzifortin}, Equation \eqref{patchproblem} has a unique and stable solution $\sigV$ that fulfills  
     \begin{align} 
        \normlo{\sigV} \cle \| \div{\w_h} \|_{\tQh'} \cle h_V \normlo{\div{\w_h}},
     \end{align}
     so property \eqref{lemmaonepropone} was shown.
     \end{proof}
\begin{remark}
  In the first step of the above estimation the constant depends on the operator norms of $\bubb_V$ und $\oswald$ which are independent of $h$. For $\oswald$ we refer to \cite{oswald},\cite{ernguermond15}. For the $\bubb_V$ using the implementation given by the coefficients \eqref{coeffbubb} the estimation is clear as $\lambda_{V_i}(x_j) \in (0,1)$.
  %, for a different implementation we refer to \cite{bubbletransform} and \cite{precondpversion}.
\end{remark}

\begin{proof}[Proof of ii. and iii.]
     Now let $c \in \R$ be a constant on the patch, then the right hand side of Equation \eqref{patchproblem} reads as
\begin{align*}
  \int_{\omega_V} \div{\w_h} \underbrace{\bubb_V (c -\oswald c}_{=0}) \intd x  = 0,
\end{align*}
but as also
\begin{align*}
  \int_{\omega_V} \div{\sig^V} c \intd x = c \int_{\partial \omega_V} \sig^V \cdot n \intd x= 0,
\end{align*}
it follows that the solution $\sig^V$ fulfills even
\begin{align*}
      \int_{\omega_V} \div{\sig^V} \psi_h \intd x  =  \skpV{\div{\w_h}, \bubb_V \left( \psi_h - \oswald \psi_h \right)}  \quad \forall \psi_h  \in \tQh(\Tc_{\ov}),
\end{align*}
in contrast to the restriction on $\tQh^0(\Tc_{\ov})$. Using a trivial extension by $\vzero$ on $\Omega \setminus \ov$ we get \eqref{lemmaoneproptwo}.
% \begin{align*}
%      \skp{\div{\sigV}, \tqh} = \skpV{\div{\w_h}, \bubb_V \left( \tqh - \oswald \tqh \right)} \quad \forall \tqh \in \tQh(\Tc).
% \end{align*}
 To show \eqref{lemmaonepropthree} we use a decomposition of the polynomial space of order $k-2$ given by
 \begin{align} \label{koszuldec}
   \begin{split}
     [\pol^{k-2}(\ov)]^2 &= \nabla\pol^{k-1}(\ov) \oplus \kos_{\vec{x}-V}(\pol^{k-3}(\ov)) \\
     [\pol^{k-2}(\ov)]^3 &= \nabla\pol^{k-1}(\ov) \oplus \kos_{\vec{x}-V}( [\pol^{k-3}(\ov)]^3), 
   \end{split}
 \end{align}
 see \cite{koszul}, Equation (3.11). Note that by the shift invariance of polynomial spaces, the origin of the Koszul operator $\kappa$ can be set to an arbitrary point $V$.
%To see this choose $p = \nabla \phi + \kos_{\vec{x}} \xi$ and shift it with $V$, so $p(x-V) = \nabla \phi(x-V) + \kos_{\vec{x}-V} \xi(x-V) =: \tilde{p}(x)$. So for a given $\tilde{p}$, set $p(x) = \tilde{p}(x+V)$ and use the representation of $p$ and define $\tilde{\phi}(x) = \phi(x-V)$ and $\tilde{\xi}(x) = \xi(x-V)$.
 For an arbitrary $b_h \in \pol^{k-1}(\ov) \subset \tQh(\Tc_{\ov})$ we get using the properties of the bubble projector \eqref{bubbprojB} and the Oswald operator
 \begin{align*}
\int_{\ov} \sigV \cdot \nabla b_h \intd x = -\int_{\ov} \div{\sigV} \b_h \intd x = -\int_{\ov} \div{\w_h} \underbrace{\bubb_V(\b_h-\oswald \b_h}_{=0}) \intd x = 0.
 \end{align*}
As $\kos_{\vec{x}-V}( [\pol^{k-3}(\ov)]^3) = W_h(\ov)$ we already know that the solution $\sigV$ of (\ref{patchproblem}) fulfills 
 \begin{align*}
   b_2(\sigV, \boldsymbol{\mu}_h) &= \int_{\ov} \sigV \cdot \kos_{\vec{x}-V}( \a_h) \intd x = 0
 \end{align*}
and so it follows \eqref{lemmaonepropthree}. For the case $d=2$ the argument is the same.
\end{proof}
\subsection{Definition of the reconstruction $\Rec$} \label{sectiondefrec}
Now we can define the reconstruction. For that we define the space
\begin{align*}
\Sigma_h := \RT{k-1}(\Tc) \subset \Hdiv{\Omega}.
\end{align*}
For a given $\w_h \in \Vh$ and all $V \in \Vert$ let $\sigV$ be the solution of Equation \eqref{patchproblem} on $\omega_{V}$ extended by $0$ on $\Omega \setminus \omega_{V}$. Then we define the reconstruction as
\begin{align} \label{defreconstruction}
  \Rec \w_h := \w_h - \sig \in V_h + \Sigma_h \quad \text{with} \quad \sig := \sum\limits_{V \in \Vert} \sigV.
\end{align}
\begin{remark}
Due to the zero normal trace of the solutions $\sigV$ on the patches $\ov$ the sum $\sig$ is still normal continuous over facets thus $\sig \in \Sigma_h$.
\end{remark}
\begin{proof}[Proof of theorem \ref{theoremone}]
  For an arbitrary $\tqh  \in \tQh$ it holds using \eqref{lemmaoneproptwo}, \eqref{bubbproj} and the properties of the bubble projector \eqref{bubbprojprop}
  \begin{align*}
\skp{\div{\Rec \w_h}, \tqh} &=\skp{\div{\w_h}, \tqh} -  \sum\limits_{V \in \Vert} (\div{\sigV}, \tqh)_{L^2(\omega_{V})} \\
                            &=\skp{\div{\w_h}, \tqh} -  \sum\limits_{V \in \Vert} (\div{\w_h}, \bubb_{V} (\tqh - \oswald \tqh))_{L^2(\omega_{V})}\\
                            &=\skp{\div{\w_h}, \tqh} -   (\div{\w_h}, \underbrace{\sum\limits_{V \in \Vert} \bubb_{V}}_{= I} (\tqh - \oswald \tqh))_{L^2(\Omega)}\\
                            &=\skp{\div{\w_h}, \tqh} -   (\div{\w_h}, \tqh)_{L^2(\Omega)} + (\div{\w_h},\oswald \tqh)_{L^2(\Omega)} \\
                            &= (\div{\w_h},\oswald \tqh)_{L^2(\Omega)}.
  \end{align*}
%  \begin{align*}
%    \skp{\div{\Rec \w_h}, \tqh} %&=\skp{\div{\w_h}, \tqh} -  \sum\limits_{V \in \Vert} \skp{\div{\sigV}, \tqh} \\
%                               &=\skp{\div{\w_h}, \tqh} -  \sum\limits_{V \in \Vert} (\div{\sigV}, \tqh)_{L^2(\omega_{V})} \\
%                               &=\skp{\div{\w_h}, \tqh} -  \sum\limits_{V \in \Vert} (\div{\w_h}, \bubb_{V} (\tqh - \oswald \tqh))_{L^2(\omega_{V})}\\
%                               %&=\skp{\div{\w_h}, \tqh} -  \sum\limits_{V \in \Vert} \sum\limits_{T \in \TcV} (\div{\w_h}, \bubb_{V,T} (\tqh - \oswald \tqh))_{L^2(T)} \\
%                               &=\skp{\div{\w_h}, \tqh} -  \sum\limits_{T \in \Tc} (\div{\w_h},  \sum\limits_{V \in \Vert_T}\bubb_{V,T} (\tqh - \oswald \tqh))_{L^2(T)}.
%  \end{align*}
%  Now we use \eqref{bubbprojprop} to get
%  \begin{align*}
%    \skp{\div{\Rec \w_h}, \tqh} &=\skp{\div{\w_h}, \tqh} -  \sum\limits_{T \in \Tc} (\div{\w_h}, \tqh - \oswald \tqh)_{L^2(T)} \\
%                                &=\skp{\div{\w_h}, \tqh} -   (\div{\w_h}, \tqh)_{L^2(\Omega)} + (\div{\w_h},\oswald \tqh)_{L^2(\Omega)} \\
%                                &= (\div{\w_h},\oswald \tqh)_{L^2(\Omega)}.
%  \end{align*}
By that it follows for an arbitrary $q_h \in Q_h$, due $\oswald q_h = q_h$, that
  \begin{align*}
    \skp{\div{(\w_h - \Rec \w_h)}, q_h} &= 0, 
  \end{align*}
%  As $\div{\Rec \w_h} \in \tQh$ we get
%  \begin{align*}
%    \norm{\div{\Rec \w_h}}{L^2(\Omega)}^2 = (\div{\w_h},\oswald \div{\Rec \w_h} )_{L^2(\Omega)}
%                                          \cle \norm{\div{\w_h}}{L^2(\Omega)} \norm{\div{\Rec \w_h}}{L^2(\Omega)}.
%  \end{align*}
 and if  $\skp{\div \w_h , q_h} = 0 ~ \forall q_h \in Q_h$ that 
  \begin{align} \label{rightoswald}
    \skp{\div{\Rec \w_h}, \tqh} = (\div{\w_h},\underbrace{\oswald \tqh}_{\in \Qh})_{L^2(\Omega)} = 0.
  \end{align}
  Finally, using \eqref{lemmaonepropthree} and \eqref{lemmaonepropone} we get 
  \begin{align*}
    \skp{\vg, \w_h - \Rec \w_h} &=  \sum\limits_{V \in \Vert}  (\vg, \sig^V)_{L^2(\omega_{V})} =  \sum\limits_{V \in \Vert}  (\vg-\Projlktwo \vg, \sig^V)_{L^2(\omega_{V})} \\
                            &\cle \sum\limits_{V \in \Vert}  \|\vg-\Projlktwo \vg\|_{L^2(\omega_{V})}  \|\sig^V\|_{L^2(\omega_{V})} \\
                            &\cle \sum\limits_{V \in \Vert}  \|\vg-\Projlktwo \vg\|_{L^2(\omega_{V})} h_V \|\div{\w_h}\|_{L^2(\omega_{V})} \\
                            &\cle \normsmooth{\vg}_{k-2}\|\nabla \w_h\|_{L^2(\Omega)}. 
  \end{align*} 
\end{proof}

\section{The Reconstruction operator for the mini finite element method} \label{section::mini}
%\todo[inline]{High order mini elemente???} Sollte so passen, LBB geht auch !
For the mini finite element method \cite{abf:1984} the bubble enriched velocity spaces read
\begin{align*}
\pol^k_{+}(T) & := \pol^k(T) \oplus \left\{ \pol^{k+d}(T) \cap H^1_0(T) \right\}, \quad \text{and}\\
  \pol_{+}^k(\Tc) & := \{ q_h : \restr{q_h}{T} \in \pol_{+}^k(T)~ \forall T \in \Tc \}.
\end{align*}
The definition of the mini element now reads as 
\begin{align*}
  \Vh := [\pol^k_{+}(\Tc)]^d \cap [C^0(\Omega)]^d \quad \text{and} \quad \Qh := \pol^k(\Tc) \cap C^0(\Omega).
\end{align*}
As in the Taylor--Hood case we solve small problems on the vertex patch $\ov$ but slightly change the right hand side and the polynomial orders. For that we define 
  \begin{align*}
  \Sigho (\Tc_{\ov}) &:= \{ \sig \in \RT{k+d-1}(\Tc_{\ov}): \tracen \sig = 0 \ondom \partial \ov  \} \subset \Hdivoz{\ov} \\
  \tQh (\Tc_{\ov}) &:= \pol^{k+d-1}(\Tc_{\ov}) \subset L^2(\ov) \qquad   \tQh^0 (\Tc_{\ov}) :=  \tQh (\Tc_{\ov}) \cap L^2_0(\ov),
\end{align*}
and for $k \ge 2$ also
\begin{align*}
   \Wh (\ov) &:= \kos_{\vec{x} - V}(\pol^{k-2}(\ov)) \subset \Lambda_V :=  \kos_{\vec{x} - V}(L^2(\ov)) \quad \text{for } \quad d=2 \\
  \Wh (\ov) &:= \kos_{\vec{x} - V}([\pol^{k-2}(\ov)]^3)  \subset  \Lambda_V := \kos_{\vec{x} - V}([L^2(\ov)]^3) \quad \text{for } \quad d=3.
\end{align*}
So for a given function $\w_h \in \Vh$ we have $\div{\w_h} \in \tQh(\Tc_{\ov})$ and seek $(\sigV, \phi_h, \boldsymbol{\lambda}_h) \in (\Sigho (\Tc_{\ov}) \times \tQh^0 (\Tc_{\ov}) \times \Wh(\ov))$ so that
\begin{align} \label{patchproblemmini}
  \blfB((\sigV, \phi_h, \boldsymbol{\lambda}_h), (\boldsymbol{\tau}_h, \psi_h, &\boldsymbol{\mu}_h)) = \skpV{\div{\w_h}, \bubb_V \left( \psi_h - \oswaldt \psi_h \right)} \\
                                                                        &\forall (\boldsymbol{\tau}_h, \psi_h, \boldsymbol{\mu}_h) \in \Sigho (\Tc_{\ov}) \times \tQh^0 (\Tc_{\ov}) \times \Wh(\ov), \nonumber
\end{align}
where $\oswaldt:\tQh(\Tc_{\ov}) \rightarrow \Qh(\Tc_{\ov})$. Note that $\oswaldt$ now maps element-wise polynomials of degree $k+d-1$ to continuous element-wise polynomials of order $k$.
\begin{remark}
  This new operator $\oswaldt$ can be seen as the Oswald operator $\oswald$ of order $k$ applied to polynomials of higher degree. 
\end{remark}
\begin{proposition}
  Equation \ref{patchproblemmini} has a unique solution $(\sigV, \phi_h, \boldsymbol{\lambda}_h)$ satisfying
  \begin{align*}
    & i.   && \normlo{\sigV} \cle h_V \normlo{\div{\w_h}},\\
    & ii.  && \skp{\div{\sigV}, \tqh} = \skpV{\div{\w_h}, \bubb_V \left( \tqh - \oswaldt \tqh \right)} \quad \forall \tqh \in \tQh(\Tc)\\
    &      && \text{where $\sigV$ was trivially extended by 0 on $\Omega$,}\\
    & iii. && \text{and the solution is $L^2(\ov)$-orthogonal on polynomials of order $k-1$, i.e.}\\
    &      && \skpV{\sigV, \xi} = 0 \quad \forall \xi \in [\pol^{k-1}(\ov)]^d.
  \end{align*}
\end{proposition}
\begin{proof}
  The proof uses exactly the same arguments as the proof of theorem \ref{localtheorem}.
\end{proof}
The reconstruction is defined as in \eqref{defreconstruction}. 
\begin{proposition}For the reconstruction operator $\Rec$ defined by \eqref{defreconstruction} holds
\begin{align}
  & i.   && \skp{\div{\Rec w_h}, \tqh} = \skp{\div{ w_h}, \oswaldt \tqh}  \quad \forall \tqh \in \tQh,\nonumber\\
  & ii.  && \skp{\div{(\vw_h - \Rec \vw_h)}, q_h} = 0 \quad \forall \vw_h \in \vV_h, \forall q_h \in \Qh,\nonumber\\
  & iii. && \skp{\div \vw_h , q_h} = 0 ~ \forall q_h \in Q_h \Rightarrow \skp{\div{\Rec w_h}, \tqh} = 0 \quad \forall \tqh \in \tQh,\nonumber\\
  &      && \hspace{6cm} \text{i.e.} \quad \div{\Rec \vw_h} = 0,\nonumber\\
  & iv.  && \skp{\vg, \vw_h - \Rec \vw_h} \le C_{\mathrm{cons}} \normsmooth{\vg}_{k-1} \|\nabla \w_h\|_{L^2(\Omega)}.\label{conserror_mini}
  \end{align}
\end{proposition}
\begin{proof}
  The proof uses exactly the same arguments as the proof of theorem \ref{theoremone}.
  In Equation \eqref{rightoswald} it is important that the Oswald operator maps to $\Qh$, which is the reason to replace $\oswald$ by $\oswaldt$ for the mini element.
\end{proof}
\begin{remark}
  The modified mini finite element method also fits in the abstract setting of Section~\ref{sec::errorest}, but
  here the consistency error is of order $k+1$ due to \eqref{conserror_mini}, i.e.
   $$
   \norm{\nabla (\vu - \vu_h)}{L^2} \leq 2 (1+ C_F) \inf_{\vw_h \in \vV_h} \norm{\nabla (\vu - \vw_h)}{L^2} + C_\mathrm{cons} \normsmooth{\Delta \vu}_{k-1}.
   $$
  Hence, also in case of the mini finite element methods,
  the pressure-dependent term from the classical estimate is replaced by a pressure-independent
  consistency error of the same order.
\end{remark}
%%% Local Variables: 
%%% mode:latex
%%% TeX-master: "siam-reconstruction"
%%% End:

%% file: rec_numex.tex
% SIAM Shared Information Template
% This is information that is shared between the main document and any
% supplement. If no supplement is required, then this information can
% be included directly in the main document.

In this section we give several numerical examples to validate and confirm the theoretical findings.
As computational framework, including the implementation of the reconstruction operator $\Rec$,
we used NGSolve (see \cite{ngsolve}) and the NGSpy interface. For all numerical examples we use unstructered, shape regular and quasi-uniform triangulations $\Tc$ generated by Netgen (see \cite{netgen}). 
\subsection{2d example}\label{sec:example1}
The first example studies the solution
\begin{align*}
  \vu := \curl{\zeta} \quad \textrm{with} \quad \zeta := x^2(x-1)^2y^2(y-1)^2 \quad \text{and} \quad
p := x^7+y^7-\frac{1}{4},
\end{align*}
of the Stokes problem on the unit square $\Omega=(0,1)^2$ with $\nu = 10^{-3}$
and the right hand side $\bfe := - \nu \laplace \vu - \nabla p$.

Tables~\ref{tab:errors_k2_example1}-\ref{tab:errors_k4_example1} show the \(L^2\) velocity and pressure errors and their estimated order
of convergence (eoc) for the modified Taylor--Hood finite element methods of order $k=2,3,4$.
All methods show the optimal convergence orders as expected by the theory.
Table~\ref{tab:errors_mini_example1} allows the same conclusions for the modified mini finite element method of lowest order.

\begin{figure}
  \centering
    \input{stokes_TH_2d_nu}
    \caption{Errors for the classical (left) and the modified (right) Taylor--Hood finite element method of order \(k=2\)
    on three fixed meshes and several choices of \(\nu\) in section~\ref{sec:example1}.\label{fig:nuerrors_example1}}
\end{figure}
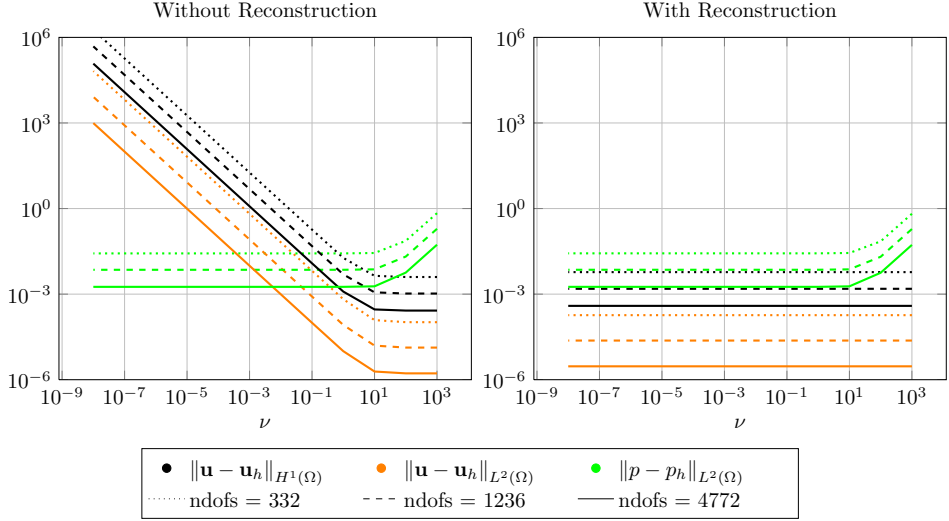

To clearly see the consequences of pressure-robustness, Figure \ref{fig:nuerrors_example1} shows the \(L^2\) errors for different
$\nu = 10^j$ for \(j=-8,\ldots,3\) on three fixed meshes for the classical and the modified Taylor--Hood finite element method of order \(k=2\).
There are several observations to make:
\begin{itemize}
 \item For \(\nu \geq 1\) the irrotational part in the right-hand side \(\bfe\) is not larger than the
 divergence-free part. In this situation both methods deliver similar errors. Due to the additional consistency error,
 the errors of the modified method are a bit larger than the errors of the classical method.
 \item For \(\nu < 1\) the irrotational part in the right-hand side \(\bfe\) begins to dominate and so does
 the pressure-dependent term in the a priori error estimate. As predicted by these estimates, the errors of the
 classical Taylor--Hood finite element method deteriorate and scale with \(1/\nu\). The modified Taylor--Hood
 method, due to its divergence-free test functions in the right-hand side, does not see the irrotational force
 and the errors are independent of \(\nu\).
 \item The transition point \(\nu \approx 1\) where the error becomes pressure-dominated
 is the same on all three meshes. Hence, mesh refinement cannot heal this behaviour.
 \item The velocity error of the modified method is independent of $\nu$, since \(\vu_h\)
 is exactly the same for every \(\nu\) by construction of the discretization. The pressure error however
 increases for large \(\nu\) in both the unmodified and the modified method. This is consistent
 with the error estimate \eqref{bestapproxpressure_estimate}.
\end{itemize}

For the mini finite element method the observations are almost identical. However, since the pressure space has the same order
as the velocity space, the pressure-dependent contributions in the a priori error estimates converge faster and can compensate
smaller values of \(\nu\) to some extent.
%To be more precise: one refinement halves the velocity best-approximation error, while
%the best-approximation error of the pressure is quartered. Hence, to again equilibrate the two errors $\nu$ can be halved.
%This explains why the transition points where the error becomes pressure-dominated moves to the left in case of the mini element, see
%Figure~\ref{fig:nuerrors_example1_mini}.
%
%\begin{figure}
%  \centering
%    \input{plots/stokes_mini_2d_nu}
%    \caption{\(L^2\) gradient errors for the classical mini finite element method of lowest order
%      on five consecutive uniform meshes and several choices of \(\nu\) in Section~\ref{sec:example1}.
%      The transition points where the error becomes pressure-dominated are marked with gray markers.\label{fig:nuerrors_example1_mini}}
%\end{figure}

\begin{table}\centering
  \pgfplotstabletypeset[
  column type=c,	%align type = r/l/c
  every head row/.style={before row=\toprule,after row=\midrule\midrule},
  every last row/.style={after row=\bottomrule},
  every even row/.style={before row={\rowcolor[gray]{0.95}}},
  % every column/.style={column type/.add={|}{|}},
  create on use/poly/.style={},
  columns/poly/.style={column name=},
  columns/ndof/.style={column name=\#dof},
  columns/h1err/.style={column name=$\|\vu-\vu_h\|_{H^1}$},
  columns/l2err/.style={column name=$\|\vu-\vu_h\|_{L^2}$},
  columns/l2errp/.style={column name=$\|p-p_h\|_{L^2}$},
  create on use/eoch1/.style={create col/dyadic refinement rate=h1err},
  columns/eoch1/.style={column name=eoc, fixed zerofill,precision=3},
  create on use/eocl2/.style={create col/dyadic refinement rate=l2err},
  columns/eocl2/.style={column name=eoc, fixed zerofill,precision=3},
  create on use/eocl2p/.style={create col/dyadic refinement rate=l2errp},
  columns/eocl2p/.style={column name=eoc, fixed zerofill,precision=3},
  % every last /.style={after row=\hline},
  % create on use/eoch1/.style={create col/expr={ (ln{\prevrow{h1err}}- ln{\thisrow{h1err}})/ln{2.0}}},
  columns = {ndof,h1err,eoch1,l2err,eocl2,l2errp,eocl2p},
  every col 3/.style={column type/.add={|}{|}}
  ]{paper_2d_p2_new.out}
  \caption{\label{tab:errors_k2_example1}Errors for the modified Taylor--Hood finite element method of order $k=2$
  in Section~\ref{sec:example1}.}
\end{table}

\begin{table}\centering
  \pgfplotstabletypeset[
  column type=c,	%align type = r/l/c
  every head row/.style={before row=\toprule,after row=\midrule\midrule},
  every last row/.style={after row=\bottomrule},
  every even row/.style={before row={\rowcolor[gray]{0.95}}},
  % every column/.style={column type/.add={|}{|}},
  columns/ndof/.style={column name=\#dof},
  columns/h1err/.style={column name=$\|\vu-\vu_h\|_{H^1}$},
  columns/l2err/.style={column name=$\|\vu-\vu_h\|_{L^2}$},
  columns/l2errp/.style={column name=$\|p-p_h\|_{L^2}$},
  create on use/eoch1/.style={create col/dyadic refinement rate=h1err},
  columns/eoch1/.style={column name=eoc, fixed zerofill,precision=3},
  create on use/eocl2/.style={create col/dyadic refinement rate=l2err},
  columns/eocl2/.style={column name=eoc, fixed zerofill,precision=3},
  create on use/eocl2p/.style={create col/dyadic refinement rate=l2errp},
  columns/eocl2p/.style={column name=eoc, fixed zerofill,precision=3},
  % every last /.style={after row=\hline},
  % create on use/eoch1/.style={create col/expr={ (ln{\prevrow{h1err}}- ln{\thisrow{h1err}})/ln{2.0}}},
  columns = {ndof,h1err,eoch1,l2err,eocl2,l2errp,eocl2p},
  every col 3/.style={column type/.add={|}{|}}
  ]{paper_2d_p3_new.out}
  \caption{\label{tab:errors_k3_example1}Errors for the modified Taylor--Hood finite element method of order $k=3$
  in Section~\ref{sec:example1}.}
\end{table}

\begin{table}\centering
  \pgfplotstabletypeset[
  column type=c,	%align type = r/l/c
  every head row/.style={before row=\toprule,after row=\midrule\midrule},
  every last row/.style={after row=\bottomrule},
  every even row/.style={before row={\rowcolor[gray]{0.95}}},
  % every column/.style={column type/.add={|}{|}},
  columns/ndof/.style={column name=\#dof},
  columns/h1err/.style={column name=$\|\vu-\vu_h\|_{H^1}$},
  columns/l2err/.style={column name=$\|\vu-\vu_h\|_{L^2}$},
  columns/l2errp/.style={column name=$\|p-p_h\|_{L^2}$},
  create on use/eoch1/.style={create col/dyadic refinement rate=h1err},
  columns/eoch1/.style={column name=eoc, fixed zerofill,precision=3},
  create on use/eocl2/.style={create col/dyadic refinement rate=l2err},
  columns/eocl2/.style={column name=eoc, fixed zerofill,precision=3},
  create on use/eocl2p/.style={create col/dyadic refinement rate=l2errp},
  columns/eocl2p/.style={column name=eoc, fixed zerofill,precision=3},
  % every last /.style={after row=\hline},
  % create on use/eoch1/.style={create col/expr={ (ln{\prevrow{h1err}}- ln{\thisrow{h1err}})/ln{2.0}}},
  columns = {ndof,h1err,eoch1,l2err,eocl2,l2errp,eocl2p},
  every col 3/.style={column type/.add={|}{|}}
  ]{paper_2d_p4_new.out}
  \caption{\label{tab:errors_k4_example1}Errors for the modified Taylor--Hood finite element method of order $k=4$
  in Section~\ref{sec:example1}.}
\end{table}

\begin{table}\centering
  \pgfplotstabletypeset[
  column type=c,	%align type = r/l/c
  every head row/.style={before row=\toprule,after row=\midrule\midrule},
  every last row/.style={after row=\bottomrule},
  every even row/.style={before row={\rowcolor[gray]{0.95}}},
  % every column/.style={column type/.add={|}{|}},
  columns/ndof/.style={column name=\#dof},
  columns/h1err/.style={column name=$\|\vu-\vu_h\|_{H^1}$},
  columns/l2err/.style={column name=$\|\vu-\vu_h\|_{L^2}$},
  columns/l2errp/.style={column name=$\|p-p_h\|_{L^2}$},
  create on use/eoch1/.style={create col/dyadic refinement rate=h1err},
  columns/eoch1/.style={column name=eoc, fixed zerofill,precision=3},
  create on use/eocl2/.style={create col/dyadic refinement rate=l2err},
  columns/eocl2/.style={column name=eoc, fixed zerofill,precision=3},
  create on use/eocl2p/.style={create col/dyadic refinement rate=l2errp},
  columns/eocl2p/.style={column name=eoc, fixed zerofill,precision=3},
  % every last /.style={after row=\hline},
  % create on use/eoch1/.style={create col/expr={ (ln{\prevrow{h1err}}- ln{\thisrow{h1err}})/ln{2.0}}},
  columns = {ndof,h1err,eoch1,l2err,eocl2,l2errp,eocl2p},
  every col 3/.style={column type/.add={|}{|}}
  ]{paper_2d_p2_mini.out}
  \caption{\label{tab:errors_mini_example1}Errors for the modified lowest-order mini finite element method in Section~\ref{sec:example1}.}
\end{table}

\subsection{3d example}\label{sec:example2}
The second example investigates the velocity and pressure
\begin{align*}
  \vu &:= \curl{(\zeta,\zeta,\zeta)} \quad \textrm{with} \quad \zeta := x^2(x-1)^2y^2(y-1)^2z^2(z-1)^2 \\
p &:= x^5+y^5+z^5-\frac{1}{2},
\end{align*}
on the unit cube $\Omega=(0,1)^3$ for $\nu = 10^{-3}$.
Table~\ref{tab:errors_k2_example2} lists the \(L^2\) errors for the modified Taylor--Hood finite element method of order
$k=2$. Also in this 3D example the convergence rates are optimal.
\begin{remark}
For the ease of implementation in NGSolve  we used Brezzi-Douglas-Marini elements of order $k$ (see \cite{brezzifortin} and \cite{Brezzi1985}) instead of the  Raviart-Thomas elements of order $k-1$ as basis for the $H(\textrm{div})$-conforming spaces $\Sigma_h(\Tc)$ and the local spaces $\Sigma_{h,0}(\Tc_{\ov})$. This does not affect the convergence order of the error. 
\end{remark}

\begin{table}\centering
  \pgfplotstabletypeset[
  column type=c,	%align type = r/l/c
  every head row/.style={before row=\toprule,after row=\midrule\midrule},
  every last row/.style={after row=\bottomrule},
  every even row/.style={before row={\rowcolor[gray]{0.95}}},
  % every column/.style={column type/.add={|}{|}},
  columns/ndof/.style={column name=\#dof},
  columns/h1err/.style={column name=$\|\vu-\vu_h\|_{H^1}$},
  columns/l2err/.style={column name=$\|\vu-\vu_h\|_{L^2}$},
  columns/l2errp/.style={column name=$\|p-p_h\|_{L^2}$},
  create on use/eoch1/.style={create col/dyadic refinement rate=h1err},
  columns/eoch1/.style={column name=eoc, fixed zerofill,precision=3},
  create on use/eocl2/.style={create col/dyadic refinement rate=l2err},
  columns/eocl2/.style={column name=eoc, fixed zerofill,precision=3},
  create on use/eocl2p/.style={create col/dyadic refinement rate=l2errp},
  columns/eocl2p/.style={column name=eoc, fixed zerofill,precision=3},
  % every last /.style={after row=\hline},
  % create on use/eoch1/.style={create col/expr={ (ln{\prevrow{h1err}}- ln{\thisrow{h1err}})/ln{2.0}}},
  columns = {ndof,h1err,eoch1,l2err,eocl2,l2errp,eocl2p},
  every col 3/.style={column type/.add={|}{|}}
  ]{paper_3d_p2_new.out}
  \caption{\label{tab:errors_k2_example2}Errors for the modified Taylor--Hood finite element method of order $k=2$
  in Section~\ref{sec:example2}.}
\end{table}

\subsection{Navier--Stokes for a 2D potential flow}
This example studies a two-dimensional potential flow for the harmonic potential \(\chi := x^5 - 10 x^3 y^2 + 5 x y^4\).
Note that $\chi$ is the real part of the analytic function $z^5$ (with $z = x + i y$).
We look for the solution of the steady incompressible Navier--Stokes equations 
$- \nu \laplace \vu + (\vu \cdot \nabla) \vu + \nabla p = \vzero$, $\div{\vu}=0$
with inhomogeneous Dirichlet boundary conditions for $\nu = 0.1$.
%\begin{align*}
%  \frac{\partial \vu}{\partial t} - \nu \laplace \vu + (\vu \cdot \nabla) \vu + \nabla p &= 0 \quad \text{in} \quad \Omega\\
%  \div{\vu}&=0 \quad \text{in} \quad \Omega \\
%             \vu &= \nabla \chi \quad \text{on} \quad \partial \Omega,
%\end{align*}
The exact solution of the velocity is given by $\vu = \nabla \chi$ and $p = 664/63 - 25/2 (x^2+y^2)^4$, modelling the collision of five jets
in the plane. For the construction and significance of potential flows the reader may consult \cite{prandtl:book}.
For the nonlinear term holds
%\begin{align*}
$(\vu \cdot \nabla) \vu = 1/2 \nabla (\vu^2)$.
%\end{align*}
Looking at the weak formulation of this term, it holds for all $\vv \in \vV^0$
\begin{align*}
\int_\Omega (\vu \cdot \nabla) \vu \cdot \vv \intd x = \int_\Omega \nabla \left( \frac{\vu^2}{2} \right) \cdot \vv \intd x = - \int_\Omega  \left( \frac{\vu^2}{2} \right)\div{\vv} \intd x = 0.
\end{align*}
This orthogonality may not hold in the discrete case, so similar as for the modified Stokes problem \eqref{disc:weak:stokes:problem},
a non-standard discretization of the nonlinear convection term is proposed that employs the reconstruction $\Rec$ in the velocity test functions 
\begin{align*}
\int_\Omega (\vu_h \cdot \nabla) \vu_h \cdot \Rec{\vv_h} \intd x.
\end{align*}
In Tables \ref{table::nvsexample} and \ref{table::nvsexamplerefine} one can see the differences in the errors,
when standard or non-standard discretizations of the nonlinear convection term are used in case of Taylor--Hood elements of order $k=2,3,4$
on two consecutive meshes with 352 and %one refinement level
1408 elements.
Note that for \(k=4\) the exact solution satisfies $\vu \in \Vh$,
but only for the non-standard discretization the velocity error vanishes.
Similar to the Stokes example \ref{sec:example1} we see that a mesh refinement does not heal the observed problems.

%\capstartfalse 
\begin{table}\centering
  \captionsetup[subfloat]{labelformat=empty}
  \subfloat[]{
  \pgfplotstabletypeset[
  column type=c,	%align type = r/l/c
  every head row/.style={before row={  
      \multicolumn{5}{c}{without reconstruction} \\
      \toprule},after row=\midrule\midrule},
  every last row/.style={after row=\bottomrule},
  every even row/.style={before row={\rowcolor[gray]{0.95}}},
  % every column/.style={column type/.add={|}{|}},
  columns/ndof/.style={column name=\#dof},
  columns/h1err/.style={column name=$\|\vu-\vu_h\|_{H^1}$},
  columns/l2err/.style={column name=$\|\vu-\vu_h\|_{L^2}$},
  columns/l2errp/.style={column name=$\|p-p_h\|_{L^2}$},
  columns = {k,ndof, h1err,l2err,l2errp},
  %every col 3/.style={column type/.add={|}{|}}
  ]{paper_2d_nvs.out}
  } \vspace{-0.5cm} \\
 \subfloat[]{
  \pgfplotstabletypeset[
  column type=c,	%align type = r/l/c
  every head row/.style={before row={  
      \multicolumn{5}{c}{with reconstruction} \\
      \toprule},after row=\midrule\midrule},
  %every head row/.style={before row=\toprule,after row=\midrule\midrule},
  every last row/.style={after row=\bottomrule},
  every even row/.style={before row={\rowcolor[gray]{0.95}}},
  % every column/.style={column type/.add={|}{|}},
  columns/ndof/.style={column name=\#dof},
  columns/h1err/.style={column name=$\|\vu-\vu_h\|_{H^1}$},
  columns/l2err/.style={column name=$\|\vu-\vu_h\|_{L^2}$},
  columns/l2errp/.style={column name=$\|p-p_h\|_{L^2}$},
  columns = {k,ndof, h1err,l2err,l2errp},
  every col 3/.style={column type/.add={|}{|}}
  ]{paper_2d_Rnvs.out}
}
\caption{Errors for the Taylor--Hood and the modified Taylor--Hood  finite element method for the Navier--Stokes example $|\Tc| = 352$}\label{table::nvsexample}
\end{table}

\begin{table}\centering
  \captionsetup[subfloat]{labelformat=empty}
  \subfloat[]{
  \pgfplotstabletypeset[
  column type=c,	%align type = r/l/c
  every head row/.style={before row={  
      \multicolumn{5}{c}{without reconstruction} \\
      \toprule},after row=\midrule\midrule},
  every last row/.style={after row=\bottomrule},
  every even row/.style={before row={\rowcolor[gray]{0.95}}},
  % every column/.style={column type/.add={|}{|}},
  columns/ndof/.style={column name=\#dof},
  columns/h1err/.style={column name=$\|\vu-\vu_h\|_{H^1}$},
  columns/l2err/.style={column name=$\|\vu-\vu_h\|_{L^2}$},
  columns/l2errp/.style={column name=$\|p-p_h\|_{L^2}$},
  columns = {k,ndof, h1err,l2err,l2errp},
  %every col 3/.style={column type/.add={|}{|}}
  ]{paper_2d_nvs_refine.out}
  } \vspace{-0.5cm} \\
 \subfloat[]{
  \pgfplotstabletypeset[
  column type=c,	%align type = r/l/c
  every head row/.style={before row={  
      \multicolumn{5}{c}{with reconstruction} \\
      \toprule},after row=\midrule\midrule},
  %every head row/.style={before row=\toprule,after row=\midrule\midrule},
  every last row/.style={after row=\bottomrule},
  every even row/.style={before row={\rowcolor[gray]{0.95}}},
  % every column/.style={column type/.add={|}{|}},
  columns/ndof/.style={column name=\#dof},
  columns/h1err/.style={column name=$\|\vu-\vu_h\|_{H^1}$},
  columns/l2err/.style={column name=$\|\vu-\vu_h\|_{L^2}$},
  columns/l2errp/.style={column name=$\|p-p_h\|_{L^2}$},
  columns = {k,ndof, h1err,l2err,l2errp},
  every col 3/.style={column type/.add={|}{|}}
  ]{paper_2d_Rnvs_refine.out}
}
\caption{Errors for the Taylor--Hood and the modified Taylor--Hood  finite element method for the Navier--Stokes example with $|\Tc| = 1408$}\label{table::nvsexamplerefine}
\end{table}

%%% Local Variables: 
%%% mode:latex
%%% TeX-master: "siam-reconstruction"
%%% End: 

%% file: stokes_TH_2d_nu.tex
\pgfplotstableread{stokes_TH21_1_2d_nu.out}\THone %564 ndofs
\pgfplotstableread{stokes_TH21_2_2d_nu.out}\THtwo %2148 ndofs
\pgfplotstableread{stokes_TH21_3_2d_nu.out}\THthree %8388 ndofs

\pgfplotstableread{stokes_RTH21_1_2d_nu.out}\RTHone %564 ndofs
\pgfplotstableread{stokes_RTH21_2_2d_nu.out}\RTHtwo %2148 ndofs
\pgfplotstableread{stokes_RTH21_3_2d_nu.out}\RTHthree %8388 ndofs
                         
\begin{tikzpicture}
  [
  %spy using outlines={rounded rectangle, width=1.5cm, height=0.15cm, very thick,black,magnification=3, connect spies}, 
  scale=0.8
  ]
  \begin{axis}[
    name=plot1,
    legend entries={$\norm{\vu-\vu_h}{H^1(\Omega)}$,
                $\norm{\vu-\vu_h}{L^2(\Omega)}$,
                $\norm{p-p_h}{L^2(\Omega)}$,
                  ndofs = 332,
                  ndofs = 1236,
                  ndofs = 4772},              
                title={Without Reconstruction},% no ghost penalty,
                scale=1,
                xlabel=$\nu$,
                legend columns=3,
                legend style={text width=8em, text height=0.9em },
    %ylabel=$$
    % ylabel=$U_B$,
    xmode=log,
    ymax=1e6,
    ymin=1e-6,
    ymode=log,
    y tick label style={
      /pgf/number format/.cd,
      fixed,
      precision=2
    },
    x tick label style={
      /pgf/number format/.cd,
      fixed,
      precision=2
    },
    % 
    % log basis x=2,
     grid=both,
    %% major grid style={black!2},
    legend style={
      cells={align=left},
      at={(0.2,-0.2)},
      anchor = north west
    },
    % xticklabel=\pgfmathparse{2^\tick}\pgfmathprintnumber{\pgfmathresult}
    ]
    \addlegendimage{only marks, mark=*,black}
    \addlegendimage{only marks, mark=*,orange}
    \addlegendimage{only marks, mark=*,green}
    \addlegendimage{dotted}
    \addlegendimage{dashed}
    \addlegendimage{}
    
    \addplot[line width=1pt, dotted, color=black] table[x=0,y=1]{\THone};
    \addplot[line width=1pt, dotted, color=orange] table[x=0,y=2]{\THone};
    \addplot[line width=1pt, dotted, color=green] table[x=0,y=3]{\THone};
    
    \addplot[line width=1pt, dashed, color=black] table[x=0,y=1]{\THtwo};
    \addplot[line width=1pt, dashed, color=orange] table[x=0,y=2]{\THtwo};
    \addplot[line width=1pt, dashed, color=green] table[x=0,y=3]{\THtwo};

    \addplot[line width=1pt, color=black] table[x=0,y=1]{\THthree};
    \addplot[line width=1pt, color=orange] table[x=0,y=2]{\THthree};
    \addplot[line width=1pt, color=green] table[x=0,y=3]{\THthree};

  \end{axis}

  \begin{axis}[
    name=plot2,
                    scale=1,
    title={With Reconstruction},% no ghost penalty, 
    xlabel=$\nu$,
    %ylabel=$$
    % ylabel=$U_B$,
    xmode=log,
    ymax=1e6,
    ymin=1e-6,
    ymode=log,
    y tick label style={
      /pgf/number format/.cd,
      fixed,
      precision=2
    },
    x tick label style={
      /pgf/number format/.cd,
      fixed,
      precision=2
    },
    % 
    % log basis x=2,
     grid=both,
     %% major grid style={black!2},
     at={(35,0)},
     anchor = south west,
    legend style={
      cells={align=left},
      at={(1.05,0.70)},
      anchor = north west
    },
    % xticklabel=\pgfmathparse{2^\tick}\pgfmathprintnumber{\pgfmathresult}
    ]

    \addplot[line width=1pt,  dotted, color=black] table[x=0,y=1]{\RTHone};
    \addplot[line width=1pt,  dotted,  color=orange] table[x=0,y=2]{\RTHone};
    \addplot[line width=1pt,  dotted, color=green] table[x=0,y=3]{\RTHone};

    \addplot[line width=1pt, dashed, color=black] table[x=0,y=1]{\RTHtwo};
    \addplot[line width=1pt, dashed,  color=orange] table[x=0,y=2]{\RTHtwo};
    \addplot[line width=1pt, dashed, color=green] table[x=0,y=3]{\RTHtwo};

    \addplot[line width=1pt, color=black] table[x=0,y=1]{\RTHthree};
    \addplot[line width=1pt,  color=orange] table[x=0,y=2]{\RTHthree};
    \addplot[line width=1pt, color=green] table[x=0,y=3]{\RTHthree};

  \end{axis}

  %% \spy[width=0.5\textwidth,height=0.15\textwidth] on (spypoint) in node[fill=white] at (spyviewer);
\end{tikzpicture}

%%% Local Variables:
%%% mode: latex
%%% TeX-master:"siam-reconstruction"
%%% End:

%% file: rec_appendix.tex
\begin{theorem} \label{app::theoremone}
For $\Omega \subseteq \R^3$, $V \in \Omega$ and $k \geq 0$ it holds
  \begin{align*}
\{ \kos_{\vec{x}-V} (\vq_1): \vq_1 \in [\pol^k(\Omega)]^3 \} = \{ \kos_{\vec{x}-V}(\vq_2): \vq_2 \in [\pol^k(\Omega)]^3, \div{\vq_2}=0\} 
  \end{align*}
\end{theorem}
\begin{proof}
  Without loss of generality we can set $V=\vzero$.
  For $k=0$ there is nothing to prove.
  In the case $k \geq 1$, for $\vq_1 \in [\pol^k(\Omega)]^3$ we define
  \begin{align*}
    \vq_2 := \vq_1 + \vec{x} w 
  \end{align*}
  with $w \in \pol^{k-1}(\Omega)$. Note that
  \begin{align} \label{app:equi}
    \kos_{\vec{x}} (\vq_2) = \vec{x} \times \vq_2 =  \vec{x} \times \vq_1 + \underbrace{\vec{x} \times \vec{x}}_{=\vzero} w = \kos_{\vec{x}} (\vq_1), 
  \end{align}
  and
  \begin{align*}
    \div{\vq_2} = \div (\vq_1 +\vec{x} w  ) =\div{\vq_1} + \div(\vec{x}) w + \vec{x} \cdot \nabla w = \div{\vq_1}  + 3 w +  \vec{x} \cdot \nabla w. 
  \end{align*}
 As we want to have $\div{\vq_2}=0$, we have to solve the equation
  \begin{align} \label{app:inhom}
 3 w +  \vec{x} \cdot \nabla w =-\div{\vq_1}.
  \end{align}
  Due to the finite dimensionality of $\pol^{k-1}(\Omega)$, this linear inhomogeneous equation can be solved, if we show that from 
  \begin{align} \label{app:assumption}
    3 w + \vec{x} \cdot \nabla w = 0 \quad \text{it follows} \quad \Rightarrow w=0.
  \end{align}
  For $k=1$ it holds $w \in \pol^0(\Omega)$ and $\vq_1 \in  [\pol^1(\Omega)]^3$ and the statement is obviously true.
  In the case $k \geq 2$ we use the following representation of $w$
  \begin{align*}
    w(x,y,z) = \tilde{w}(x,y,z) + \sum\limits_{i=0}^{k-1} \sum\limits_{j=0}^{k-1-i}c_{ij} x^i y^j z^{k-1-i-j}, 
  \end{align*}
  with $\tilde{w} \in \pol^{k-2}(\Omega)$. Using assumption (\ref{app:assumption}) and $3 \tilde{w} + \vec{x} \cdot \nabla \tilde{w} =: \hat{w} \in \pol^{k-2}(\Omega)$  we now have
  %Now it holds
  %\begin{align*}
  %  3 w + \vec{x} \cdot \nabla w  = \underbrace{3 \tilde{w} + \vec{x} \cdot \nabla \tilde{w}}_{=: \hat{w} \in \pol^{k-2}(\Omega)} &+ \sum\limits_{j=0}^{k-1} \sum\limits_{i=0}^{k-1-j} 3 c_{ij} x^jy^{k-j}z^{k-j-i}\\
  %                                                                                                                                &+ \sum\limits_{j=0}^{k-1} \sum\limits_{i=0}^{k-1-j} x c_{ij}j x^{j-1}y^{k-j}z^{k-j-i}\\
  %                                                                                                                                &+ \sum\limits_{j=0}^{k-1} \sum\limits_{i=0}^{k-1-j} y c_{ij}(k-j) x^jy^{k-j-1}z^{k-j-i}\\
  %                                                                                                                                    &+ \sum\limits_{j=0}^{k-1} \sum\limits_{i=0}^{k-1-j} z c_{ij}(k-j-i) x^jy^{k-j}z^{k-j-i-1}                              
  %\end{align*}
  %so
  %\begin{align*}
  %  3 w + \vec{x} \cdot \nabla w = \underbrace{3 \tilde{w} + \vec{x} \cdot \nabla \tilde{w}}_{=: \hat{w} \in \pol^{k-2}(\Omega)} + \sum\limits_{j=0}^{k-1} \sum\limits_{i=0}^{k-1-j} \underbrace{(3+j+k-j+k-j-i)}_{>0} c_{ij} x^jy^{k-j}z^{k-j-i}. 
  %\end{align*}
% e that
  \begin{align*}
    3 w + \vec{x} \cdot \nabla w  = \hat{w} + \sum\limits_{i=0}^{k-1} \sum\limits_{j=0}^{k-1-i} (k-1) c_{ij} x^i y^j z^{k-1-i-j} = 0 \\
    \quad \forall (x,y,z) \in \Omega,
  \end{align*}
  what leads to $c_{ij}=0 ~\forall i,j$ and so $\hat{w} = 3 \tilde{w} + \vec{x} \cdot \nabla \tilde{w} = 0$. By induction it follows $w = 0$.
  Therefore, we can solve equation (\ref{app:inhom}) and for every $\vq_1$ we find a $\vq_2$ with $\div{\vq_2}=0$ and due to (\ref{app:equi}) the theorem is shown.
\end{proof}
%%% Local Variables: 
%%% mode:latex
%%% TeX-master: "siam-reconstruction"
%%% End: